\UseRawInputEncoding
\documentclass[a4paper,10pt]{article}
\usepackage{graphicx}
\usepackage{subfigure}
\usepackage{epstopdf}
\usepackage{booktabs}
\usepackage{threeparttable}
\usepackage{amssymb, amsmath, amsthm, epsfig}
\numberwithin{equation}{section}
\usepackage{latexsym, enumerate}
\usepackage{eepic}
\usepackage{epic}
\usepackage{color}

\usepackage{amsmath,textcomp, amssymb, mathrsfs, bm, amsthm, amsfonts}
\usepackage{algorithm}
\usepackage{algpseudocode}
\allowdisplaybreaks[4]
\usepackage{ifpdf}

\usepackage{chngcntr}

\usepackage{dsfont}
\usepackage{multirow}
\usepackage{bm}
\usepackage{caption}

\theoremstyle{plain}
\newtheorem{lem}{Lemma}[section]
\newtheorem{thm}[lem]{Theorem}
\newtheorem{cor}[lem]{Corollary}

\newtheorem{property}{Property}
\theoremstyle{definition}
\newtheorem{defn}{Definition}[section]
\newtheorem{assum}{Assumption}[section]

\theoremstyle{remark}
\newtheorem{rem}{Remark}[section]

\renewcommand{\thefigure}{\thesection.\arabic{figure}}

\begin{document}
\renewcommand{\figurename}{Figure}
\renewcommand{\thesubfigure}{(\alph{subfigure})}
\renewcommand{\thesubtable}{(\alph{subtable})}
\makeatletter
\renewcommand{\p@subfigure}{\thefigure~}

\makeatother
\title{\large\bf A Hadamard fractioal total variation-Gaussian (HFTG) prior for Bayesian inverse problems}
\author{
Li-Li Wang\thanks
{School of Mathematics, Hunan University, Changsha 410082, China.
Email: lilywang@hnu.edu.cn}
\and
Ming-Hui Ding\thanks
{School of Mathematics, Hunan University, Changsha 410082, China.
Email: minghuiding@hnu.edu.cn}
\and
Guang-Hui Zheng\thanks
{School of Mathematics, Hunan University, Changsha 410082, China. Email: zhenggh2012@hnu.edu.cn (Corresponding author)}
}
\date{}
\maketitle

\begin{center}{\bf ABSTRACT}
\end{center}\smallskip
This paper studies the infinite-dimensional Bayesian inference method with Hadamard fractional total variation-Gaussian (HFTG) prior for solving inverse problems. First, Hadamard fractional Sobolev space is established and proved to be a separable Banach space under some mild conditions. Afterwards, the HFTG prior is constructed in this separable fractional space, and the proposed novel hybrid prior not only captures the texture details of the region and avoids step effects, but also provides a complete theoretical analysis in the infinite dimensional Bayesian inversion. Based on the HFTG prior, the well-posedness and finite-dimensional approximation of the posterior measure of the Bayesian inverse problem are given, and samples are extracted from the posterior distribution using the standard pCN algorithm. Finally, numerical results under different models indicate that the Bayesian inference method with HFTG prior is effective and accurate.

\smallskip
{\bf keywords}: Bayesian inference; Hadamard fractional total variation; hybrid prior; pCN algorithm

\section{Introduction}\label{sec1}
Motivated by significant scientific and industrial applications, the research in theories and computational methods of inverse problems has undergone tremendous growth in the last decades. Consider the inverse problem of find $u$ from $y$, where $u$ and $y$ are related by
\begin{equation}\label{bayesian model}
y=G(u)+\eta,
\end{equation}
here, $G: X \to \mathbb{R}^{m}$ is a measurable mapping known as forward operator, $X$ is a separable Hilbert space with inner product $\langle\cdot ,\cdot \rangle_{ X }$, $u\in X$ is the unknown function, $y\in\mathbb{R}^{m}$ is the finite-dimensional observed data, and $\eta$ is an $m$-dimensional Gaussian noise with zero mean and covariance operator as $\Sigma$, namely, $\eta\thicksim \mathcal{N}(0,\Sigma)$.

Bayesian inference methods have attracted much attention in various applications of inverse problems since they constitute a complete probabilistic description of the inverse problem and therefore provide a natural framework for quantifying its uncertainty \cite{Wang2004, Kaipio2005}.

This framework involves characterizing the posterior distribution of parameters in terms of given a prior distribution and observed data, together with a forward model connecting the space of parameters to the data space. Assuming the prior measure of $u$ is $\mu_{pr}$, the posterior distribution on $u$ is given by measure $\mu^y$ satisfying
\begin{equation}\label{HybridPriorMeasure}
\frac{d\mu^{y}}{d\mu_{pr}}(u) \propto\exp(-\Phi(u)),
\end{equation}
where the left hand side of (20) is the Radon-Nikodym derivative of the posterior
distribution $\mu^y$ with respect to the prior $\mu_0$ and
\begin{equation}\label{potential}
\Phi(u):=\frac12\big\|G(u)-y\big\|^2_\Sigma=\frac12\big\|\Sigma^{-1/2}(G(u)-y)\big\|_2^2,
\end{equation}
is potential function in Bayesian theory which often referred to as the data fidelity term in deterministic inverse problems. Formula \eqref{HybridPriorMeasure} holds in infinite dimensions, in practical term, this means that posterior expectations can be found by reweighting prior expectations by the right-hand side of Eq \eqref{HybridPriorMeasure}.

 A prior selection plays a key role in Bayesian inference methods for solving inverse problems, and a suitable prior distribution can significantly improve the inference results. In the existing work, Gaussian prior is the most commonly used prior distribution, such as \cite{Dashti2013,Stuart2015,Stuart2010}. However, the Gaussian prior is less effective for sharp jumps or discontinuities in the inversion, in order to overcome this shortcoming, a hybrid prior of the Bayesian inference method is proposed. The idea of this hybrid prior is described as follows: assume the Gaussian measure is $\mu_0$, and the prior measure $\mu_{pr}$ is absolutely continuous with respect to the $\mu_0$, i.e.,
 \begin{equation}\label{PriorMeasure}
\frac{d\mu_{pr}}{d\mu_{0}}(u) \propto\exp(-R(u)),
\end{equation}
where $R(u)$ represents additional prior information (or regularization) on $u$. It is easy to see that, under this assumption, the Radon-Nikodym derivative of $\mu^y$ in (1.2) satisfying
\begin{equation}\label{hybridR-N}
\frac{d\mu^y}{d\mu_0}(u) \propto \exp(-\Phi(u)-R(u)),
\end{equation}
which is interpreted as the Bayesian formula with hybrid prior, and also returns to the conventional formulation with Gaussian priors. In \cite{Z.Yao2016}, the hybrid total variation-Gaussian (TG) prior is proposed for the first time, which selects a TV regularization term  as the additional regularization term $R$ to detect edges. Nevertheless, the TG prior has the similar disadvantage as TV regularization method in \cite{Vogel2002}, which will cause issues such as, stair effects, texture loss and over smooth etc.

Recently, the fractional total variation regularization method is widely used in imaging inverse problems, for examples, image denoising, deblurring (or deconvolution), registration, super-resolution \cite{chan2019,yao2020,zhang2011,Zhang2012,DL.Chen2013,Williams2016,Ruiz2009,Zhang2015,Chen2015,Zhou2017,Ren2013}. They can ease the conflict between staircase elimination and edge preservation by choosing the order of derivative properly compared with TV. Moreover, the fractional derivative operator has a ¡°non-local¡± behavior because the fractional derivative at a point depends upon the characteristics of the entire function and not just the values in the vicinity of the point \cite{Pdlubny1999,Kilbas2006,Samko1993}, which is beneficial to improve the performance of texture preservation. The numerical results in literatures \cite{D.Chen2013,Ren2013,Zhang2015,Chen2015} have demonstrated that the fractional derivative performs well for eliminating staircase effect and preserving textures. Motivated by these works, we propose to consider the fractional total variation-Gaussian (FTG) prior, which takes the fractional total variation (FTV) regularization term as the additional prior information $R(u)$. This hybrid prior not only allows for flexible recovery of texture and geometric patterns for various imaging inverse problems, but also uses the Gaussian distribution as a reference measure to ensure that the resulting prior converge to a well-defined probability measure in the function space in the limit of infinite dimensionality. Among the available works, the studies on the FTV terms are all based on the classical Riemann-Liouville or Gr\"{u}nwald-Letnikov fractional derivative, while in this article, we investigate the FTV term based on the Hadamard fractional derivative, more definitions and details are present in Section 2.

 In this work, we investigates applying the Bayesian inference method with the  Hadamard fractional total variation-Gaussian (HFTG) prior for solving the inverse problem of (1.1). We first give the basic definitions and properties of Hadamard fractional derivative, the HFTV term and fractional Sobolev space. Then, the well-posedness and finite-dimensional approximation of the posterior distribution which is derived from the Bayesian method with HFTG prior, is proved theoretically in infinite dimension setting. Finally, it is verified that our proposed prior is robust and valid by reconstructing the image using the standard preconditioned Crank-Nicolson (pCN) algorithm and giving different numerical examples. To provide a global view of our study, the major contributions of this work can be summarised as follows.
\begin{itemize}
\item{We propose to use the HFTG prior of Bayesian inference method for image reconstruction. This hybrid prior, on the one hand, preserves detailed information about the images, and on the other hand allows to build a theoretical analysis in the infinite-dimensional Bayesian inference framework. Afterwards, the fractional Sobolev space for the corresponding HFTG prior is constructed and proved to be a separable space, which is essential to establish probabilities and integrals theory in the infinite dimensional Bayesian method.}
\item{We investigate the nature of the posterior distribution of the Bayesian approach based on the HFTG prior. It reveals the well-posedness framework of the inverse problem and the numerical approximation to the posterior measure, which verify the discretization-invariant (or dimension-independent) \cite{bui2016,lassas2004} property of our algorithm}.
\item{Finally, according to the smoothness at different regions of images, the true images are reconstructed by using different fractional orders in HFTG prior, the reconstruction results are substantial improvement, thus verifying the robustness and effectiveness of our proposed priror.}
\end{itemize}
The outline of this paper is given as follows. In section 2, we provide preliminary knowledge on definitions and some basic properties of the Hadamard fractional derivative. In section 3, we build the HFTG prior and give some common properties of the posterior distribution based on the Bayesian framework with HFTG prior. Sections 4 and 5 are respectively devoted to numerical experiments and conclusion.
%
%
\section{Preliminaries}
In this section we present the definitions and some properties of the Hadamard fractional derivative. In particular, the Hadamard fractional order derivative is a special case of the general Riemann-Liouville fractional derivative in \cite{Samko1993,Kilbas2006}.

\begin{defn}\label{Hadamard fractional derivative}Let $\Omega=[a,b]$ be a interval with $0<a<b$, and $f(x)$ be a function defined in $\Omega$. For $\alpha \in (n-1,n)$, $n\in {\mathds{N}}^{+}$, the left Hadamard fractional derivative of $f(x)$ of order $\alpha$ is given by
\begin{equation}
^H\negmedspace D^{\alpha}_{[a,x]}f(x)=\frac{1}{\Gamma(n-\alpha)}\left(x\frac{d}{dx}\right)^n\int_a^x\left(\ln\frac{x}{t}\right)^{n-\alpha-1}\frac{f(t)}{t}dt,
\end{equation}
and the right Hadamard fractional derivative of $f(x)$ of order $\alpha$ is given by
\begin{equation}
^H\negmedspace D^{\alpha}_{[x,b]}f(x)=\frac{1}{\Gamma(n-\alpha)}\left(-x\frac{d}{dx}\right)^n\int_x^b\left(\ln\frac{t}{x}\right)^{n-\alpha-1}\frac{f(t)}{t}dt,
\end{equation}
meanwhile, the Riesz (or central) Hadamard fractional derivative is given by
\begin{equation}^{RH}\negmedspace D^{\alpha}_{[a,b]}f(x):=\frac{1}{2}\left(^{H}\negmedspace D^{\alpha}_{[a,x]}f(x)+(-1)^n\,{^H}\negmedspace  D^{\alpha}_{[x,b]}f(x)\right),
\end{equation}
where, $\Gamma(x)$ represents Gamma function.
\end{defn}

Next, we evoke another Hadamard fractional derivative, which is motivated by the definition of the classical Caputo fractional derivative given by \cite{Samko1993,Kilbas2006,Pdlubny1999}.
\begin{defn} \label{Caputo fractional derivative}
Let $\Omega=[a,b]$ be a interval with $0<a<b$, and $f(x)$ be a function defined in $\Omega$. For $\alpha \in (n-1,n)$, $n\in {\mathds{N}}^{+}$, the left Caputo fractional derivative of $f(x)$ of order $\alpha$ is given by
\begin{equation}
^C\! D^{\alpha}_{[a,x]}f(x):=\frac{1}{\Gamma(n-\alpha)}\int_a^x\left(\ln\frac{x}{t}\right)^{n-\alpha-1}\left(t\frac{d}{dt}\right)^n \negmedspace f(t)\frac{dt}{t}
\end{equation}
and the right Caputo fractional derivative of $f(x)$ of order $\alpha$ is given by
\begin{equation}
{^C}\!D^{\alpha}_{[x,b]}f(x):=\frac{1}{\Gamma(n-\alpha)}\int_x^b\left(\ln\frac{t}{x}\right)^{n-\alpha-1}\left(-t\frac{d}{dt}\right)^n \negmedspace f(t)\frac{dt}{t}.
\end{equation}
Meanwhile, the Riesz-Caputo fractional derivative is given by
$${^{RC}}\! D^{\alpha}_{[a,b]}f(x):=\frac{1}{2}\left({^C}\!D^{\alpha}_{[a,x]}f(x)+(-1)^n\,
{^C}\!D^{\alpha}_{[x,b]}f(x)\right).$$
\end{defn}
To simplify notation, we will use the abbreviated differential operator form
$$D_{H}=x\frac{d}{dx},D_{H}^n=\underbrace{D_{H}\cdot D_{H} \cdots D_{H} }_{n\  times}.$$

Despite that the above two definitions are different from each other, there is a relationship between them as \cite{Almeida2017,Jarad2020,Sousa2018}.
\begin{thm}\label{relationship}
If $f \in C^n(\Omega)$ and $\alpha>0$, then
\begin{align*}
{^C}\!D^{\alpha}_{[a,x]}f(x)&={^H}\negmedspace D^{\alpha}_{[a,x]}\left[f(x)-\sum^{n-1}_{k=0}\frac{D_{H}^kf(a)}{k!}\left(\ln\frac{x}{a}\right)^k \right]\\
&={^H}\negmedspace D^{\alpha}_{[a,x]}f(x)-\sum^{n-1}_{k=0}\frac{D_{H}^k f(a)}{\Gamma(k-\alpha+1)}\left(\ln\frac{x}{a}\right)^{k-\alpha},
\end{align*}
and
\begin{align*}
{^C}\!D^{\alpha}_{[x,b]}f(x)&={^H}\negmedspace D^{\alpha}_{[x,b]}\left[f(x)-\sum^{n-1}_{k=0}\frac{(-1)^{k}D_{H}^k f(b)}{k!}\left(\ln\frac{b}{x}\right)^k \right]\\
&={^H}\negmedspace D^{\alpha}_{[x,b]}f(x)-\sum^{n-1}_{k=0}\frac{(-1)^{k}D_{H}^k f(b)}{\Gamma(k-\alpha+1)}\left(\ln\frac{b}{x}\right)^{k-\alpha}.
\end{align*}
\end{thm}

\begin{rem}\label{equivalence}(\textbf{Equivalence})
Let $f \in C^n(\Omega)$, for all $k=0, \dots, n-1$, if $D_{H}^k  f(a)=0$, we have
\begin{equation}\label{relation at a}
{^C}\!D^{\alpha}_{[a,x]}f(x)={^H}\negmedspace D^{\alpha}_{[a,x]}f(x),
\end{equation}
and if $D_{H}^k  f(b)=0$, we deduce
\begin{equation}\label{relation at b}
{^C}\!D^{\alpha}_{[x,b]}f(x)={^H}\negmedspace D^{\alpha}_{[x,b]}f(x).
\end{equation}
From the definitions of Riesz fractional derivative, show that
\begin{equation}\label{relation riesz}
{^{RC}}\!D^{\alpha}_{[a,b]}f(x)={^{RH}}\negmedspace D^{\alpha}_{[a,b]}f(x).
\end{equation}
Thus, under the above conditions, the Hadamard fractional derivatives are equivalent to the Hadamard-Caputo fractional derivatives.
\end{rem}

In addition, there is a common property for the fractional differential operator.
\begin{property}\label{pro:linearity}(\textbf{Linearity})
Let $\mathcal{P}$ denote the fractional calculus operator, $k, l\in \mathbb{R}$ are constants, for any fractional differentiable functions $f(x)$ and $g(x)$, we have:
\[\mathcal{P}(kf(x)+lg(x))=k\mathcal{P}(f(x))+l\mathcal{P}(g(x)).\]
\end{property}

The next, we will establish fractional integration by parts formula similarly as \cite{Almeida2017}, which is useful to derive the property of fractional Sobolev space.
\begin{thm}\label{fractional integration by parts formula} (\textbf {fractional integration by parts formula})
Given $f \in C(\Omega)$ and \\$g \in C^n (\Omega)$, we have that for all $\alpha>0$,
\begin{equation}\label{left fractional integration by parts formula}
\begin{split}
\int^b_ax^{-1}f(x)\,{^C}\!D^{\alpha}_{[a,x]}g(x)dx
=&\int^b_ax^{-1}\,{^H}\negmedspace D^{\alpha}_{[x,b]}f(x)\,g(x)dx\\
+&\sum^{n-1}_{k=0}\left[{^H}\negmedspace D^{\alpha-n+k}_{[x,b]}f(x)\, D_{H}^{n-k-1}g(x)\right]^{x=b}_{x=a},
\end{split}
\end{equation}
and
\begin{equation}\label{right fractional integration by parts formula}
\begin{split}
\int^b_ax^{-1}f(x)\,{^C}\!D^{\alpha}_{[x,b]}g(x)dx
=&\int^b_ax^{-1}\,^H\negmedspace D^{\alpha}_{[a,x]}f(x)\,g(x)dx\\
+&\sum^{n-1}_{k=0}\left[(-1)^{n+k}\,{^H}\negmedspace D^{\alpha-n+k}_{[a,x]}f(x)\, D_{H}^{n-k-1}g(x)\right]^{x=b}_{x=a}.
\end{split}
\end{equation}
Then,
\begin{equation}\label{Riesz fractional integration by parts formula}
\begin{split}
\int^b_ax^{-1}f(x)\,{^{RC}}\!D^{\alpha}_{[a,b]}g(x)dx
=&(-1)^n\int^b_ax^{-1}\,{^{RH}}\negmedspace D^{\alpha}_{[a,b]}f(x)\,g(x)dx\\
+&\sum^{n-1}_{k=0}\left[(-1)^{k}\,{^{RH}}\negmedspace D^{\alpha-n+k}_{[a,b]}f(x)\,D_{H}^{n-k-1}g(x)\right]^{x=b}_{x=a}.
\end{split}
\end{equation}
\end{thm}
\begin{proof}
Now, if we can prove the first equation is true, the second equation will be true similarly, furthermore the third equation holds with the definition of the Riesz fractional derivative. Using the definition of Hadamard fractional derivative and Dirichlet's formula, we first compute
\begin{align*}
\int^b_ax^{-1}f(x)\,{^C}\!D^{\alpha}_{[a,x]}g(x)dx
=&\frac{1}{\Gamma(n-\alpha)}\int^b_ax^{-1}f(x)\int_a^x\left(\ln\frac{x}{t}\right)^{n-\alpha-1}D_{H}^nf(t)\,\frac{dt}{t}dx\\
=&\frac{1}{\Gamma(n-\alpha)}\int^b_a t^{-1}D_{H}^ng(t)\int_t^b\left(\ln\frac{x}{t}\right)^{n-\alpha-1} \frac{f(x)}{x}dxdt\\
=&\frac{1}{\Gamma(n-\alpha)}\int^b_a\int_x^b\left(\ln\frac{t}{x}\right)^{n-\alpha-1} \frac{f(t)}{t}dt\cdot\frac{d}{dx}D_{H}^{n-1}g(x)dx,
\end{align*}

by applying integration by parts for the right side of the last equation, we can get
\begin{align*}
&\left[\frac{1}{\Gamma(n-\alpha)}\int_x^b\left(\ln\frac{t}{x}\right)^{n-\alpha-1} \frac{f(t)}{t}dt\cdot D_{H}^{n-1}g(x)\right]^{x=b}_{x=a}\\
&\ \ \ -\frac{1}{\Gamma(n-\alpha)}\int^b_a\frac{d}{dx}\left(\int_x^b\left(\ln\frac{t}{x}\right)^{n-\alpha-1} \frac{f(t)}{t}dt\right)D_{H}^{n-1} g(x)dx\\
=&\left[\frac{1}{\Gamma(n-\alpha)}\int_x^b\left(\ln\frac{t}{x}\right)^{n-\alpha-1} \frac{f(t)}{t}dt\cdot D_{H}^{n-1}g(x)\right]^{x=b}_{x=a}\\
&\ \ \ +\frac{1}{\Gamma(n-\alpha)}\int^b_a\left(-x\frac{d}{dx}\right)\left(\int_x^b\left(\ln\frac{t}{x}\right)^{n-\alpha-1} \frac{f(t)}{t}dt\right)\cdot \frac{d}{dx}D_{H}^{n-2} g(x)dx.
\end{align*}
Let us apply integration by parts once more, the last formula is equal to
\begin{align*}
&\sum^1_{k=0}\left[\frac{1}{\Gamma(n-\alpha)}\left(-x\frac{d}{dx}\right)^k \int_x^b\left(\ln\frac{t}{x}\right)^{n-\alpha-1} \frac{f(t)}{t}dt\cdot D_{H}^{n-k-1}g(x)\right]^{x=b}_{x=a}\\
&\ \ \ +\frac{1}{\Gamma(n-\alpha)}\int^b_a\left(-x\frac{d}{dx}\right)^2\left(\int_x^b\left(\ln\frac{t}{x}\right)^{n-\alpha-1} \frac{f(t)}{t}dt\right)\cdot \frac{d}{dx}D_{H}^{n-3} g(x)dx.
\end{align*}
Repeating the process, we have
\begin{align*}
&\sum^{n-1}_{k=0}\left[\frac{1}{\Gamma(n-\alpha)}\left(-x\frac{d}{dx}\right)^k \int_x^b\left(\ln\frac{t}{x}\right)^{n-\alpha-1} \frac{f(t)}{t}dt\cdot D_{H}^{n-k-1}g(x)\right]^{x=b}_{x=a}\\
&\ \ \ +\frac{1}{\Gamma(n-\alpha)}\int^b_a\left(-x\frac{d}{dx}\right)^n\left(\int_x^b\left(\ln\frac{t}{x}\right)^{n-\alpha-1} \frac{f(t)}{t}dt\right)\cdot g(x) \,x^{-1}dx\\
=&\sum^{n-1}_{k=0}\left[^H\negmedspace D^{\alpha-n+k}_{[x,b]}f(x)\,D_{H}^{n-k-1} g(x)\right]^{x=b}_{x=a}+\int^b_a x^{-1}\,^H\negmedspace D^{\alpha}_{[x,b]}f(x)\,g(x)dx.
\end{align*}
Consequently
\begin{align*}
\int^b_ax^{-1}f(x)\,{^C}\!D^{\alpha}_{[a,x]}g(x)dx
=&\int^b_ax^{-1}\,{^H}\negmedspace D^{\alpha}_{[x,b]}f(x)\,g(x)dx\\
+&\sum^{n-1}_{k=0}\left[{^H}\negmedspace D^{\alpha-n+k}_{[x,b]}f(x)\, D_{H}^{n-k-1}g(x)\right]^{x=b}_{x=a}.
\end{align*}

Similarly, we can prove that the second equation is true, i.e.,
\begin{align*}
\int^b_ax^{-1}f(x)\,{^C}\!D^{\alpha}_{[x,b]}g(x)dx
=&\int^b_ax^{-1}\,^H\negmedspace D^{\alpha}_{[a,x]}f(x)\,g(x)dx\\
+&\sum^{n-1}_{k=0}\left[(-1)^{n+k}\,{^H}\negmedspace D^{\alpha-n+k}_{[a,x]}f(x)\, D_{H}^{n-k-1}g(x)\right]^{x=b}_{x=a}.
\end{align*}

Finally, using the definition of Riesz fractional derivative and the equations \eqref{left fractional integration by parts formula} and \eqref{right fractional integration by parts formula}, we see that the third equation \eqref{Riesz fractional integration by parts formula} holds.
\end{proof}

Obviously, if for all $k=0, \dots, n-1$, $D_{H}^k g(a)=0$ and $D_{H}^k g(b)=0$, combining with equations \eqref{relation at a}, \eqref{relation at b} and \eqref{relation riesz}, then equations \eqref{left fractional integration by parts formula}, \eqref{right fractional integration by parts formula} and \eqref{Riesz fractional integration by parts formula} in Theorem \ref{fractional integration by parts formula} will become
\begin{equation}\label{left parts equation}
\int^b_ax^{-1}f(x)\,{^H}\negmedspace D^{\alpha}_{[a,x]}g(x)dx
=\int^b_ax^{-1}\,{^H}\negmedspace D^{\alpha}_{[x,b]}f(x)\,g(x)dx,
\end{equation}
\begin{equation}\label{right parts equation}
\int^b_ax^{-1}f(x)\,{^H}\negmedspace D^{\alpha}_{[x,b]}g(x)dx
=\int^b_ax^{-1}\,^H\negmedspace D^{\alpha}_{[a,x]}f(x)\,g(x)dx,
\end{equation}
and
\begin{equation}\label{riesz parts equation}
\int^b_ax^{-1}f(x)\,{^{RH}}\!D^{\alpha}_{[a,b]}g(x)dx
=(-1)^n\int^b_ax^{-1}\,{^{RH}}\negmedspace D^{\alpha}_{[a,b]}f(x)\,g(x)dx.
\end{equation}

In subsequent papers, to distinguish the definitions, we use $^C\!D^{\alpha}$ and $D^{\alpha}$ to represent the fractional derivative based on Caputo and Hadamard fractional derivative respectively.

%
%
\section{The HFTG prior}\label{sec3}
In this section, based on the Bayesian framework with hybrid prior to inverse problems in section 1,  we will specify the space of unknown functions $X$ and the additional regularization term $R$, then construct the HFTG prior.
%
%
\subsection{The Hadamard Fractional Total Variation}\label{FTV}\label{sec3.1}

This subsection first studies the fractional Sobolev space and prove the separablility, which plays an important role in the development of probability and integration in infinite dimensional spaces. Second define the fractional total variation based on the Hadamard fractional derivative.

\begin{defn}\label{sobolevspace}(\textbf{Fractional Sobolev Space}) For any positive integer $p\in\mathbf{N}^+$, let $$W^{\alpha}_{p}(\Omega)=\{u\in L^p(\Omega)\big|\|u\|_{W^{\alpha}_{p}(\Omega)}<+\infty\}$$ be a fractional Sobolev function space endowed with the norm
$$\|u\|_{W^{\alpha}_{p}(\Omega)}=\left(\int_a^b|u|^p dx+\int_a^b|D_{[a,b]}^{\alpha} u|^p dx\right)^{\frac{1}{p}}.$$
\end{defn}
Specially, when $p=2$, the above norm is generated by the following inner product$$\langle u,v\rangle_{W^{\alpha}_{2}(\Omega)}=\int_a^buvdx+\int_a^b(D_{[a,b]}^{\alpha}u)(D_{[a,b]}^{\alpha}v)dx,\ \ u,v \in W^{\alpha}_{2}(\Omega).$$


Then, before discussing the total fractional-order variation, we give the following definition as \cite{Zhang2015,Chen2015}, which based on the equivalence in Remark \ref{equivalence}.
\begin{defn}\label{space of test functions}(\textbf{Spaces of test functions})
Denote by $\mathcal{C}^{\ell}(\Omega,\mathcal{R}^d)$ the space of an $\ell$-order continuously differentiable functions in $\Omega \subset \mathcal{R}^d$. Then an $\ell$-order compactly supported continuous function space as a subspace $\mathcal{C}^{\ell}(\Omega,\mathcal{R}^d)$ is denoted by $\mathcal{C}^{\ell}_{0}(\Omega,\mathcal{R}^d)$, in which each member $v:\Omega\mapsto\mathcal{R}^d$ satisfies the homogeneous boundary conditions $D_{H}^iv(x)|_{\partial\Omega}=0$ for all $i=0, 1, \dots, \ell$.
\end{defn}
With a test function $g(x)\in \mathcal{C}^{n}_{0}(\Omega,\mathcal{R})$, the $\alpha$-order integration by parts formulas can also be rewritten as equations \eqref{left parts equation}, \eqref{right parts equation} and \eqref{riesz parts equation}.

Next, we can prove that the fractional Sobolev space is a separable Banach space with $1\leqslant p<\infty$ following by the ideas of classical Sobolev space as \cite{Brezis2011,Evans1998} and \cite{Agrawal2007,Bourdin2015,Idczak2013} .
\begin{lem}\label{Banachspace}
The fractional Sobolev space $W^{\alpha}_{p}(\Omega)$ is a Banach space.
\end{lem}
\begin{proof}
(1) First, we should testify that the $\|u\|_{W^{\alpha}_{p}(\Omega)}$ is a norm. By the definitions of $\|u\|_{L^p(\Omega)}$ and the fractional derivative $D_{[a,b]}^{\alpha}u$ with the linearity, we can easily prove
\[\|qu\|_{W^{\alpha}_{p}(\Omega)}=|q|\|u\|_{W^{\alpha}_{p}(\Omega)},\] and
\[\|u\|_{W^{\alpha}_{p}(\Omega)}=0 \ if\  and\  only\  if\  u=0 \ a.e.\]
Next, assume $u,v \in W^{\alpha}_{p}(\Omega)$ and $1\leqslant p<\infty$, according to the Minkowski's inequality, we can obtain
\begin{align*}
\|u+v\|_{W^{\alpha}_{p}(\Omega)}
=&\left(\|u+v\|_{L^{p}(\Omega)}^{p}+\|D_{[a,b]}^{\alpha} u+D_{[a,b]}^{\alpha} v\|_{L^{p}(\Omega)}^{p}\right)^{\frac{1}{p}}\\
\leqslant&[(\|u\|_{L^{p}(\Omega)}+\|v\|_{L^{p}(\Omega)})^{p}+(\|D_{[a,b]}^{\alpha} u\|_{L^{p}(\Omega)}+\|D_{[a,b]}^{\alpha} v\|_{L^{p}(\Omega)})^{p}]^{\frac{1}{p}}\\
\leqslant&\left(\|u\|_{L^{p}(\Omega)}^{p}+\|D_{[a,b]}^{\alpha} u\|_{L^{p}(\Omega)}^{p}\right)^{\frac{1}{p}}+\left(\|v\|_{L^{p}(\Omega)}^{p}+\|D_{[a,b]}^{\alpha} v\|_{L^{p}(
\Omega)}^{p}\right)^{\frac{1}{p}}\\
=&\|u\|_{W^{\alpha}_{p}(\Omega)}+\|v\|_{W^{\alpha}_{p}(\Omega)}.
\end{align*}
Hence, $W^{\alpha}_{p}(\Omega)$ is a norm space.

(2) Then, it only need to prove the completeness of $W^{\alpha}_{p}(\Omega)$. Suppose $\left\{u_m\right\}_{m=1}^{\infty}\subset{W^{\alpha,p}_{\psi}(\Omega)}$ is a Cauchy sequence, following the definition of norm, $\left\{u_m\right\}_{m=1}^{\infty}$ and $\left\{D_{[a,b]}^{\alpha}u_m\right\}_{m=1}^{\infty}$ are both Cauchy sequence in $L^p(\Omega)$. According to the completeness of $L^p(\Omega)$, there exist two functions $u$ and $u^{\alpha}$ in $L^p(\Omega)$ such that $$u_m \rightarrow u,\ \  D_{[a,b]}^{\alpha}u_m \rightarrow u^{\alpha},\ \  in\  L^p(\Omega).$$
For any $g(x)\in \mathcal{C}^{n}_{0}(\Omega,\mathcal{R})$, we discover
\begin{align*}
\int^b_ax^{-1}\,u\,D_{[a,b]}^{\alpha}gdx=&\lim_{m\rightarrow +\infty}\int^b_a x^{-1}\,u_m\,D_{[a,b]}^{\alpha}gdx\\
=&\lim_{m\rightarrow +\infty}(-1)^n\int_a^bx^{-1}D_{[a,b]}^{\alpha}u_m\,gdx\\
=&(-1)^n\int_a^bx^{-1}\,u^{\alpha}\,gdx\\
=&(-1)^n\int_a^bx^{-1}D_{[a,b]}^{\alpha}u\,gdx.
\end{align*}
Thus $u^{\alpha}=D_{[a,b]}^{\alpha}u$, $u \in W^{\alpha}_{p}(\Omega)$, and $\|u_m-u\|_{W^{\alpha}_{p}(\Omega)}\rightarrow 0$ when $m\to\infty$. Conclusion, $W^{\alpha}_{p}(\Omega)$ is a Banach space.\\
\end{proof}

\begin{lem}\label{separablespace}
For all $1\leqslant p<\infty$, the space $W^{\alpha}_{p}(\Omega)$ is a separable space.
\end{lem}
\begin{proof}
Let us consider the product space $(L^p)^2=L^p\times L^p$ endowed with the norm $$\|(u,v)\|_{(L^p)^2}=(\|u\|_{L^p}^p+\|v\|_{L^p}^p)^{\frac{1}{p}}.$$
Since $1 \leqslant p <\infty$, the space $(L^p,\| \cdot \|_{L^p})$ is a separable space, therefor $((L^p)^2,\| \cdot \|_{(L^p)^2})$ is also a separable space. We define $O:=\left\{(u,D_{[a,b]}^{\alpha}u)\big|u \in W^{\alpha}_{p}(\Omega)\right\}$. Obviously, $O$ is a subspace of $((L^p)^2,\| \cdot \|_{(L^p)^2})$ and then $O$ is a separable space. Finally, defining the following mapping
\begin{align*}
T: W^{\alpha}_{p}(\Omega)&\rightarrow O\subset (L^p)^2 \\
u&\mapsto(u,D_{[a,b]}^{\alpha}u).
\end{align*}
We can prove that the mapping $T$ is one to one, and $$\|T(u)\|_{(L^p)^2}=\|u\|_{W^{\alpha}_{p}(\Omega)}.$$
Consequently, the mapping $T$ is isometric isomorphic to $O$, and then $W^{\alpha}_{p}(\Omega)$ is separable space with respect to $\| \cdot \|_{W^{\alpha}_{p}(\Omega)}$. \\
\end{proof}

Thus, when $p=2$, $W^{\alpha}_{2}(\Omega)$ is a separable Hilbert space.
\begin{lem}\label{lem:embedding}
The following embedding result holds: $$W^{\alpha}_{2}(\Omega)\subset W^{\alpha}_{1}(\Omega).$$
\end{lem}
\begin{proof}
For any $u \in W^{\alpha}_{2}(\Omega)$, and using H\"{o}lder's inequality, we deduce
\begin{equation*}
\begin{split}
\|u\|^2_{W^{\alpha}_{1}}(\Omega)=&\left(\int_a^b|u|dx+\int_a^b|D_{[a,b]}^{\alpha}u|dx\right)^2\\
\leqslant&\left[\left(\int_a^b|u|^2dx\right)^{\frac{1}{2}}\left(\int_a^b 1 dx\right)^{\frac{1}{2}}+
\left(\int_a^b|D_{[a,b]}^{\alpha}u|^2dx\right)^{\frac{1}{2}}\left(\int_a^b 1 dx\right)^{\frac{1}{2}}\right]^2\\
=&C\left[\left(\int_a^b|u|^2dx\right)^{\frac{1}{2}}+\left(\int_a^b|D_{[a,b]}^{\alpha}u|^2dx\right)^{\frac{1}{2}}\right]^2\\
\leqslant&C\left[\int_a^b|u|^2dx+\int_a^b|D_{[a,b]}^{\alpha}u|^2dx\right]\\
=&C\|u\|^2_{W^{\alpha}_{2}}(\Omega),
\end{split}
\end{equation*}
where C is only dependent on the size of $\Omega$. We therefore conclude $\|u\|_{W^{\alpha}_{1}}(\Omega)\leqslant C\|u\|_{W^{\alpha}_{2}}(\Omega)$, i.e., $u\in W^{\alpha}_{1}(\Omega)$ or $W^{\alpha}_{2}(\Omega)\subset W^{\alpha}_{1}(\Omega)$.\\
\end{proof}

In a conclusion, we can choose $$X=W^{\alpha}_{2}(\Omega),$$ and total fractional variation as the additional regularization term $R$, i.e.,
\begin{equation}\label{R definition}
R(u)=\lambda \|u\|_{HFTV}=\lambda \int_a^b|D_{[a,b]}^{\alpha}u|dx,
\end{equation}
where $\lambda$ is the regularization parameter.
%
%
\subsection{Theoretical properties of the HFTG prior}\label{sec3.2}
In this subsection, we discuss the well-posedness and the finite dimensional approximation of the posterior distribution arising from the Bayesian inference with HFTG hybrid prior. The proofs are similar as \cite{Z.Yao2016}, thus we omits proofs here.

We assume that the forward operator $G:W^{\alpha}_{2}(\Omega)\rightarrow \mathbb{R}^m$ satisfies the following assumptions as \cite{Stuart2010}:
\begin{assum}\label{forward operator assume} \ \\
(i) for every $\varepsilon > 0$ there exists $M=M(\varepsilon) \in \mathbb{R}$ such that, for all $u \in W^{\alpha}_{2}(\Omega)$,
$$\|G(u)\|_{\Sigma}\leqslant \exp (\varepsilon \|u\|^{2}_{W^{\alpha}_2}(\Omega)+M),$$
(ii) for every $ r>0$, there exists $K=K(r)>0$ such that, for all $u_1,u_2 \in W^{\alpha}_2(\Omega)$ with $\max\left\{\|u_1\|_{W^{\alpha}_2(\Omega)},\|u_2\|_{W^{\alpha}_2(\Omega)}\right\}<r$,$$\big\|G(u_1)-G(u_2)\big\|_{\Sigma}\leqslant K\|u_1-u_2\|_{W^{\alpha}_2(\Omega)}.$$
\end{assum}

The Assumptions \ref{forward operator assume}. about $G$ can derive the bounds and Lipschitz properties of $\Phi$ as Assumptions 2.6 in \cite{Stuart2010}.
We shall show that the HFTG prior is well-behaved, i.e., there is a lemma about the additional regularization term $R$ should holds as following.
\begin{lem}\label{R property}
Let $R:W^{\alpha}_{2}(\Omega)\rightarrow \mathbb{R}^m$ defines as equation (\ref{R definition}). Then $R$ satisfies the followings:\\
$(i)$ For all $u \in W^{\alpha}_{2}(\Omega)$, we have $R(u)\geqslant 0$.\\
$(ii)$ For every $r>0$, there exists $K=K(r)>0$ such that, for all $u \in W^{\alpha}_{2}(\Omega)$ with $\|u\|_{W^{\alpha}_{2}(\Omega)}<r,R(u)\leqslant K$.\\
$(iii)$ For every $r>0$, there exists $L=L(r)>0$ such that, for all $u_1,u_2 \in W^{\alpha}_2(\Omega)$ with $\max\left\{\|u_1\|_{W^{\alpha}_2(\Omega)},\|u_2\|_{W^{\alpha}_2(\Omega)}\right\}<r$,$$|R(u_1)-R(u_2)|\leqslant L\|u_1-u_2\|_{W^{\alpha}_2(\Omega)}.$$
\end{lem}
\begin{proof}
$(i)$: This property is trivial.\\
$(ii)$: For all $u \in W^{\alpha}_{2}(\Omega)$ and given $r>0$ with $\|u\|_{W^{\alpha}_{2}(\Omega)}<r$, from the Lemma \ref{lem:embedding}, there exists a constant $C>0$ such that $$R(u)=\lambda \|u\|_{HFTV}\leqslant \lambda \|u\|_{W^{\alpha}_1(\Omega)}\leqslant \lambda C \|u\|_{W^{\alpha}_2(\Omega)}\leqslant \lambda C r.$$
So, we can choose $K(r)=\lambda C r$.\\
$(iii)$: For every $r>0$ and all $u_1,u_2 \in W^{\alpha}_2(\Omega)$, there is constant $C>0$ such that $$|R(u_1)-R(u_2)|=\lambda\|u_1-u_2\|_{HFTV}\leqslant \lambda \|u_1-u_2\|_{W^{\alpha}_1(\Omega)}\leqslant \lambda C \|u_1-u_2\|_{W^{\alpha}_2(\Omega)}.$$
When we choose $L(r)=\lambda C$, $(iii)$ is proved. This proof is complete.\\
\end{proof}

If $\Phi$ satisfies Assumptions 2.6 in \cite{Stuart2010} and Lemma \ref{R property} holds respect to $R$, we can obtain that $\Phi+R$ satisfies Assumptions 2.6 in \cite{Stuart2010}. As a conclusion, the probability measure $\mu^{y}$ given by equation (\ref{hybridR-N}) is well defined on $W^{\alpha}_2(\Omega)$ and it is Lipschitz in the data $y$ with respect to the Hellinger metric as following theorem.
\begin{thm}\label{welldefined}
Let forward operator $G:W^{\alpha}_{2}(\Omega)\rightarrow \mathbb{R}^m$ satisfies Lemma \ref{R property} and $R:W^{\alpha}_{2}(\Omega)\rightarrow \mathbb{R}^m$ is defined as (\ref{R definition}). For a given $y\in \mathbb{R}^m$,  $\mu^{y}$ is given by equation (\ref{hybridR-N}). Then we have the following:\\
$(i)$ $\mu^{y}$ defined as equation (\ref{hybridR-N}) is well defined on $W^{\alpha}_2(\Omega)$.\\
$(ii)$ $\mu^{y}$ is Lipschitz in the data $y$ with respect to the Hellinger metric. specifically, if $\mu^{y}$ and $\mu^{y'}$ are two measures corresponding to data $y$ and $y'$ respectively, then for every $r>0$, there exists $C=C(r)>0$ such that for all $y$, $y'\in \mathbb{R}^m$ with $\max\left\{\|y\|_{W^{\alpha}_{2}(\Omega)},\|y'\|_{W^{\alpha}_{2}(\Omega)}\right\}<r$, we have $$d_{Hell}(\mu^{y},\mu^{y'}) \leqslant C \|y-y'\|_{\Sigma}.$$ Consequently the expectation of any polynomially bounded function $f:W^{\alpha}_{2}(\Omega) \to E$ is
continuous in $y$.
Where, $E$ is the Cameron-Martin space of the Gaussian measure $\mu_{0}$ (see \cite{Stuart2010,Stuart2015}),
and the Hellinger metric with respect to meaasure $\mu$ and $\mu'$ is defined by
$$d_{Hell}(\mu, \mu')=\sqrt{\frac12\int_{\Omega}\big(\sqrt{\frac{d\mu}{d\nu}}-\sqrt{\frac{d\mu'}{d\nu}}\big)^2 d\nu}.$$
\end{thm}

The theorem as above is direct consequence of the fact that $\Phi+R$ satisfies the assumption (2.6) in \cite{Stuart2010} and so we omit the proof here. So far, theoretically, we have shown the validity of the HFTG prior. Next we will study the feasibility of the HFTG prior numerically and the finite dimensional approximation of posterior measure $\mu^y$ is essential. In particular, we consider the following approximation in \cite{Z.Yao2016}:
\begin{equation}\label{approximation 1}
\frac{d\mu^{y}_{N_1,N_2}}{d\mu_0} \propto \exp (-\Phi_{N_1}(u)-R_{N_2}(u)),
\end{equation}
where, $\Phi_{N_1}(u)$ is a $N_1$ dimensional approximation of $\Phi(u)$ with $G_{N_1}$ being the $N_1$ dimensional approximation of forward operator $G$ and $R_{N_2}(u)$ is a $N_2$ dimensional approximation of $R(u)$. Then we can establish a approximation of posterior distribution $\mu^y$ to $\mu^y_{N_1,N_2}$ with respect to Hellinger metric. Since the proofs of the following Theorem \ref{approximationtheorem} and Corollary \ref{approximationcor} are similar as the ones in \cite{Z.Yao2016}. To make this paper self-contained, we give their proofs in the Appendix.
\begin{thm}\label{approximationtheorem}
Assume that $G$ and $G_{N_1}$ satisfy assumption \ref{forward operator assume} (i) with constants uniform in $N_1$, and $R$ and $R_{N_2}$ satisfy Lemma \ref{R property} (i) and (ii) with constants uniform in $N_2$. Assume further for all $\varepsilon>0$, there exist two sequences $\{a_{N_1}(\varepsilon)\}>0$ and  $\{b_{N_2}(\varepsilon)\}>0$ which are both converge to zero, such that $\mu_0(X_{\varepsilon}) \geqslant 1-\epsilon$, for all $N_1, N_2$,
\begin{equation}
X_{\varepsilon}=\left\{u \in W^{\alpha}_{2}(\Omega)\big||\Phi(u)-\Phi_{N_1}(u)| \leqslant {a_{N_1}(\varepsilon)},|R(u)-R_{N_2}(u)| \leqslant {b_{N_2}(\varepsilon)}\right\},
\end{equation}
Then we can obtain
\begin{equation}
d_{Hell}(\mu^{y},\mu^{y}_{N_1,N_2}) \to 0,\  as\  N_1,N_2 \to +\infty.
\end{equation}
\end{thm}

Noting that $W^{\alpha}_{2}(\Omega)$ is a separable Hilbert space, hence we can show the finite dimensional approximation of $\mu^y$ to $\mu^y_N$ without any additional assumptions, as following:
\begin{cor}\label{approximationcor}
Let $\{e_k\}_{k=1}^{\infty}$ be a complete orthogonal basis of $W^{\alpha,\psi}_{2}(\Omega)$. For all $N\in \mathbb{N}$, we define
\begin{equation}\label{approximationcor 1}
u_{N}=\sum_{k=1}^N \langle u,e_{k}\rangle e_{k}
\end{equation}
and
\begin{equation}\label{approximationcor 2}
\frac{d\mu^{y}_{N}}{d\mu_0}=\exp \left(-\Phi(u_{N})-R(u_{N})\right).
\end{equation}
Assume $G$ satisfies Assumption \ref{forward operator assume} and $R$ is defined by (\ref{R definition}), then
\begin{equation}\label{approximationcor 3}
d_{Hell}(\mu^{y},\mu^{y}_{N}) \to 0,\ as \ N \to \infty.
\end{equation}
\end{cor}


\subsection{A discretization of the Hadamard fractional derivative}\label{sec3.3}
In this subsection, we discuss the behavior of this HFTG prior in numerically. In the numerical performance, the approximation of Hadamard fractional derivative is essential, and we present a discretization for the Hadamard fractional derivative based on the relationship between the classical Riemann-Liouville and Hadamard fractional derivative as following.
\begin{rem}\label{general classical}
In \cite{Mohamed2017}, we can know that there is a relationship between the classical Riemann-Liouville fractional derivative for left
\begin{equation}
^{RL}\! D^{\alpha}_{[a,x]}f(x):=\frac{1}{\Gamma(n-\alpha)}\left(\frac{d}{dx}\right)^n\int_a^x\frac{f(t)dt}{(x-t)^{\alpha-n+1}},
\end{equation}
and the general Riemann-Liouville fractional derivative for left
\begin{equation}\label{general}
D^{\alpha,\psi}_{[a,x]}f(x):=
\frac{1}{\Gamma(n-\alpha)}\left(\frac{1}{\psi'(x)}\frac{d}{dx}\right)^n\int_a^x \frac{\psi'(t)f(t)dt}{(\psi(x)-\psi(t))^{\alpha-n+1}},
\end{equation}
which is described as following:
$$D_{[a,x]}^{\alpha,\psi}f(x)=D_{[a,x]}^{\alpha,\psi}f(\psi^{-1}(s))={^{RL}}\!D_{[\psi(a),s]}^{\alpha}(f\circ\psi^{-1})(s),\ as\  x=\psi^{-1}(s), s\in[\psi(a),\psi(b)].$$
The right fractional derivative has the similar relationship as the left:
$$D_{[x,b]}^{\alpha,\psi}f(x)=D_{[x,b]}^{\alpha,\psi}f(\psi^{-1}(s))={^{RL}}\!D_{[s,\psi(b)]}^{\alpha}(f\circ\psi^{-1})(s),\ as\  x=\psi^{-1}(s), s\in[\psi(a),\psi(b)].$$
Then,
\begin{equation}\label{relationRL}
D_{[a,b]}^{\alpha,\psi}f(x)=D_{[a,b]}^{\alpha,\psi}f(\psi^{-1}(s))={^{RL}}\!D_{[\psi(a),\psi(b)]}^{\alpha}(f\circ\psi^{-1})(s).
\end{equation}
When $\psi(x)=\ln(x)$ in equation \eqref{general}, it will become the Hadamard fractional derivative.
\end{rem}

Similar as the approximate approach of classical Riemann-Liouville fractional derivative in \cite{Q.Yang2010}, based on the relationship, the Hadamard fractional derivative can be approximated by followings. Let $s_{j}=\psi(a)+jh, j=0,1,\dots,N$ with $x_{j}=\psi^{-1}(s_{j})$ and the step size $h=\frac{\psi(b)-\psi(a)}{N}$, and let $f_{j}$ be the approximation to $(f\circ\psi^{-1})(s_{j})$. When $0<\alpha< 1$ and $n=1$, the general Riemann-Liouville fractional derivative is approximated as follows:
$$D_{[a,x]}^{\alpha,\psi}f(x_l)={^{RL}}\!D_{[\psi(a),s]}^{\alpha}(f\circ\psi^{-1})(s_{l})\thickapprox\frac{1}{h^{\alpha}}\sum_{j=0}^l \omega_jf_{l-j},$$
and$$D_{[x,b]}^{\alpha,\psi}f(x_l)={^{RL}}\!D_{[s,\psi(b)]}^{\alpha}(f\circ\psi^{-1})(s_{l})\thickapprox\frac{1}{h^{\alpha}}\sum_{j=0}^{N-l} \omega_jf_{l+j},$$
thus,
\begin{equation}\label{priordiscret1}
D_{[a,b]}^{\alpha,\psi}u(x_l)={^{RL}}\!D_{[\psi(a),\psi(b)]}^{\alpha}(f\circ\psi^{-1})(s_{l})\thickapprox\frac{1}{2h^{\alpha}}\left(\sum_{j=0}^l \omega_jf_{l-j}-\sum_{j=0}^{N-l} \omega_jf_{l+j}\right).
\end{equation}
When $1<\alpha< 2$ and $n=2$, the general Riemann-Liouville fractional derivative is approximated as follows:
$$D_{[a,x]}^{\alpha,\psi}f(x_l)={^{RL}}\!D_{[\psi(a),s]}^{\alpha}(f\circ\psi^{-1})(s_{l})\thickapprox\frac{1}{h^{\alpha}}\sum_{j=0}^{l+1} \omega_jf_{l-j+1},$$
and$$D_{[x,b]}^{\alpha,\psi}f(x_l)={^{RL}}\!D_{[s,\psi(b)]}^{\alpha}(f\circ\psi^{-1})(s_{l})\thickapprox\frac{1}{h^{\alpha}}\sum_{j=0}^{N-l+1} \omega_jf_{l+j-1},$$
therefore,
\begin{equation}\label{priordiscret2}
D_{[a,b]}^{\alpha,\psi}u(x_l)={^{RL}}\!D_{[\psi(a),\psi(b)]}^{\alpha}(f\circ\psi^{-1})(s_{l})\thickapprox\frac{1}{2h^{\alpha}}\left(\sum_{j=0}^{l+1} \omega_jf_{l-j+1}+\sum_{j=0}^{N-l+1} \omega_jf_{l+j-1}\right).
\end{equation}
Where $l=1,2,\dots,N-1$, $\omega_{0}=1,\omega_{j}=(-1)^j\frac{\alpha(\alpha-1)\dots(\alpha-j+1)}{j!}$, for $j=1,2,\dots,N$. In fact, the coefficient has recursion formula as following:
\[\omega_{0}=1, \ \omega_{j}=\left(1-\frac{1+\alpha}{j}\right)\omega_{j-1}\ \text{for} \ j>0.\]
When we take the $\psi(x)=\ln(x)$ in equation \eqref{general}, it will become the Hadamard fractional derivative. Meanwhile, specify the $\psi(x)=\ln(x)$ in approximate scheme of the general Riemann-Liouville fractional derivative, we can obtain the approximation for the Hadamard fractional derivative.

%
%
\section{Numerical algorithm and experiments}\label{sec4}
This section we will discuss the numerical implementation method of the Bayesian inference with respect to HFTG prior to capture the information from the posterior measure and the applications of this prior numerically.
%
%
\subsection{pCN algorithm}\label{sec4.1}
In general it is hard to obtain information from a probability measure in
high dimensions. One useful approach to extracting information is to find
a maximum a posteriori estimator, or MAP estimator, the other commonly used method for interrogating a probability measure in high dimensions is sampling. Markov chain Monte Carlo (MCMC) methods are widely used sampling methods in the Bayesian inference.
In the algorithm implementation, we use the preconditioned Crank-Nicolson (pCN) MCMC algorithm to drawn samples from the posterior distribution $\mu^y$ given by equation \eqref{hybridR-N}, which developed in \cite{Stuart2015,Stuart2010,Stuart2013}, due to its dimension-independent properties. Therefore, for high dimensions, pCN algorithm provides a more robust and efficient technique than the standard MCMC approaches. Following is a brief introduction to pCN algorithms:

Give the propose by
\begin{equation}\label{propose}
v=\sqrt{1-\beta^2}u+\beta w,
\end{equation}
where $v$ is the next propose situation, $u$ is the current situation, $\beta$ is a positive constant, and $w\thicksim \mathcal{N}(0,\mathcal{C}_0)$.

The associated acceptance probability is given by
\begin{equation}\label{acceptance}
a(u,v)=\min\{1,\exp[\Phi(u)+R(u)-\Phi(v)-R(v)]\}.
\end{equation}
The Algorithm \ref{algorithm} is the detailed description of the pCN algorithm process.
\begin{algorithm}[ht]
\caption{The preconditioned Crank-Nicolson (pCN) Algorithm}
\label{algorithm}
\begin{algorithmic}[1]
\State {Initialize $u^{(0)}\in W^{\alpha}_{2}(\Omega)$;}
\For{$i=0$ to $n$}
	\State {Propose $v^{(i)}=\sqrt{1-\beta^2}u^{(i)}+\beta w^{(i)}, w^{(i)}\thicksim \mu_{0}$;}
	\State{Draw $\theta \thicksim U[0,1]$}
	\If{$\theta\leqslant a(u^{(i)},v^{(i)})$}
		\State {$u^{(i+1)}=v^{(i)}$;}
	\Else
		\State {$u^{(i+1)}=u^{(i)}$;}
	\EndIf
\EndFor
\end{algorithmic}
\end{algorithm}

%
%
\subsection{Numerical Examples}\label{sec4.2}
In this subsection, we show some numerical results obtained by the HFTG prior for three types examples, two of which are linear problem, deconvolution problem and inverse source identification problems, and the other is nonlinear problem, the parameter identification by interior measurements. Then the salient and promising features of the HFTG prior can be illustrated from the results. All the numerical simulations are interest in $0<\alpha<1$ and $1<\alpha<2$, i.e., $n=1$ or $n=2$ in the definitions of fractional derivatives. Furthermore, we only consider four numerical results for the fractional fractional order $\alpha=0.1,\  0.9,\  1.1$, and $1.9$, and compare them to the results of TG prior. Here, the regularization parameter $\lambda$ is manually chosen so that we obtain the optimal inversion results. Besides, averaging the estimate results of many times running, may reduce the erroneous influence which randomness brings.

%
%
\subsubsection{A Deconvolution Problem}\label{sec4.2.1}
The first problem is a simple deconvolution problem in image processing problem as \cite{Vogel2002}. Consider the Fredholm first kind integral equation of convolution type:
\begin{equation}\label{convelution}
g(x)=\int_{\Omega}k(x-x')f(x')dx'\overset{def}=(\mathcal{K}f)(x), \ x\in\Omega,
\end{equation}
here $g$ represents the blurred image, $f$ represents source term. The kernel $k$ is given by following Gaussian kernel,
\begin{equation}
k(x)=C \exp(-x^2/2r^2),
\end{equation}
where $C$ and $r$ are positive parameters with $C=1/(r\sqrt{2\pi})$.

The associated inverse problem is as following: Given the kernel $k$ and the blurred image $g$, determine the source $f$. In this example, given $\Omega=[1,2]$, the source $f$ is defined by
$$ f(x)=\left\{
\begin{array}{ll}
-16(x-1)(x-1.5),  &{1\leqslant x\leqslant1.5;}\\
0.5,  &{1.7\leqslant x \leqslant 1.9;}\\
0, &{otherwise.}
\end{array} \right. $$
We can simply  discretize equation \eqref{convelution} to obtain a discrete linear system $Kf=d$ by using left rectangle formula on a uniform grid in $s=\ln(x)$ with $N=100$, $s_{j}=jh, j=0,1,\dots,N$ with $x_{j}=\exp(s_{j})$ as section \ref{sec3.3}, and the $K$ has entries
$$[K]_{ij}=h C\exp\left(-\frac{(x_{i}-x_{j})^2}{2r^2}\right),1\leq i,j\leq N,$$
here, fixed $r=0.03$. The noisy measured data $y$ are generated by
\begin{equation}\label{data}
y=Kf+\eta,
\end{equation}
where $\eta$ is the Gaussian random vector with a zero mean and 0.01 standard deviation.

Specifically, we choose the reference Gaussian prior to be $\mathcal{N}(0,\mathcal{C}_{0})$ and the covariance operator $\mathcal{C}_{0}$ is given by
\begin{equation}\label{cov}
c_{0}(x_1,x_2)=\gamma\exp\left[-\frac{1}{2}\left(\frac{x_1-x_2}{d}\right)^2\right],
\end{equation}
where $\gamma=0.01$ and $d=0.02$ in the subsequent numerical experiment.
For the FTG prior, we should use a finite dimensional formula and assume the prior density is
\begin{equation}\label{FOTV density}
p(f_N)\propto\exp(-\lambda\|f_N\|_{HFTV}),
\end{equation}
and the finite dimensional approximation for the Hadamard fractional derivative in $\|f_N\|_{FTV}$ with the equations \eqref{priordiscret1} and \eqref{priordiscret2} for fractional order $0<\alpha<1$ and $1<\alpha<2$. Moreover, we fix $\beta = 0.03$ and extract $ 2 \times10 ^ 5 $ samples from all posterior measure in the iterative process of pCN algorithm.
\begin{figure}[ht]
\centering
\subfigure[]{
\label{fig1:subfig1}
\includegraphics[width=7.5cm]{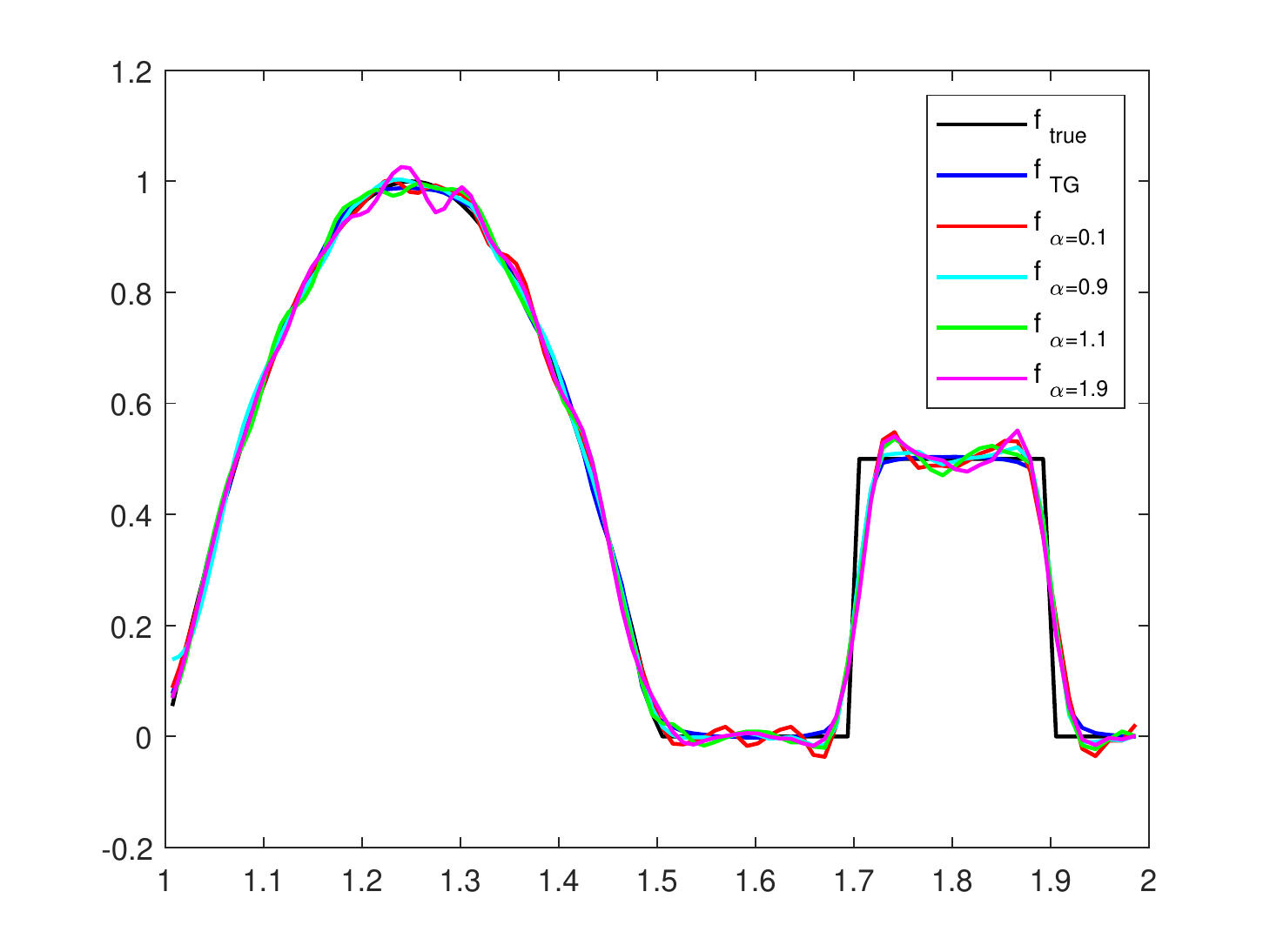}
}
\quad
\subfigure[]{
\label{fig1:subfig2}
\includegraphics[width=7.5cm]{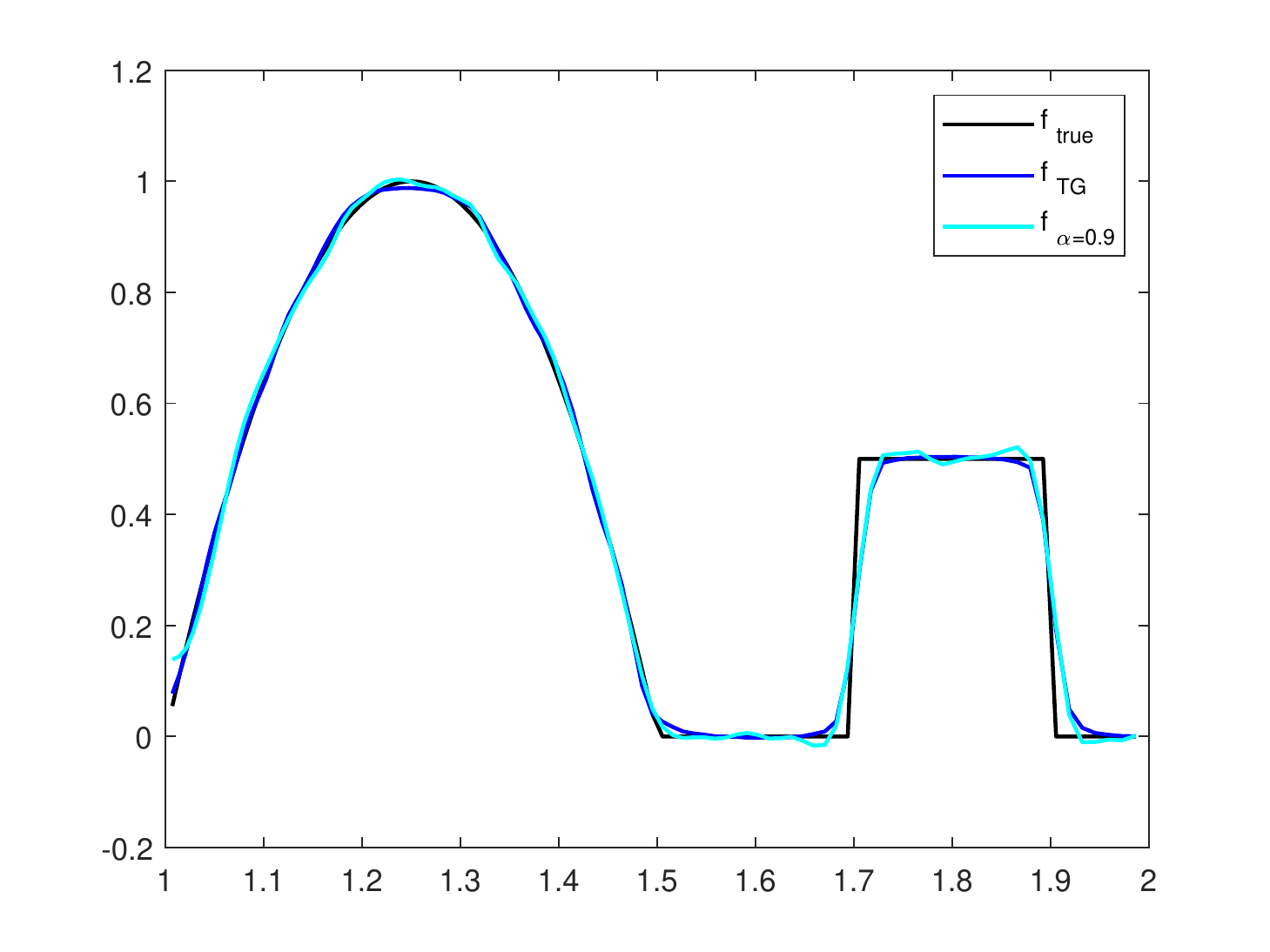}
}
\caption{(a): The true solution and inversion solution: The legend $f_{TG}$ represents TG prior inversion results, legend $f_{true}$ represents the true solution, and the others represents HFTG prior inversion results with different fractional oreder $\alpha$. (b): HFTG prior with $\alpha=0.9$ compared with TG prior.}
\label{figure1}
\end{figure}

The numerical results for reconstructing the source term $f$ with different fractional order $\alpha$ calculated by the HFTG prior and TG prior are shown in Figure \ref{figure1}. For the HFTG prior, when $\alpha=0.1,\ 0.9,\ 1.1$, and $1.9$, the corresponding regularization parameters are $\lambda=0.01,\ 2,\ 0.1$, and $0.0001$ respectively, and for the TG prior, we choose $\lambda=2$. The parameters in the Figure \ref{fig1:subfig2} are the same as in the Figure \ref{fig1:subfig1} with $\alpha=0.9$ and $f_{TG}$. We can see that the HFTG prior with various $\alpha$ and TG prior are well approximations of the exact solution, which also indicates that the HFTG and TG prior is valid in a deconvolution problem.

It is worth noting that, from the Figure \ref{fig1:subfig2}, the results of HFTG prior with $\alpha=0.9$ are basically the same as the results of TG prior, while the results of $\alpha=1.1$ is different. These results are consistent with the findings reported in \cite{Kilbas2006,Samko1993}, that is when $\alpha=n\in \mathbb{N}$, the classical left and right Riemann-Liouville fractional derivatives are consistent with integer $n$ order derivative $f^{(n)}(x)$ and $-f^{(n)}(x)$. Hence when $0<\alpha<1$ as $\alpha\to 1_{-}$ and $n=1$, the Riesz Riemann-Liouville fractional derivative is consistent with $f^{'}(x)$, while when $1<\alpha<2$ as $\alpha\to 1_{+}$ and $n=2$, will not so. There is the same reason of this phenomenon because of the relationship between Hadamard fractional derivative and classical Riemann-Liouville fractional derivative. Meanwhile, only when $\alpha=1$, the scheme \eqref{priordiscret1} degenerates into the central difference for $f^{'}(x)$. When $\alpha=0.1$ and $1.9$ are far from $1$, the results are more smooth compared with the result of TG prior.

From the Table \ref{tab1}, we can see that the errors for the three different $N$ look almost identical, suggesting that the results with the TG prior and FTG prior are independent of discretization dimensionality.
\renewcommand{\arraystretch}{1.5}
\begin{table}[htbp]
  \centering
  \caption{The errors of TG prior and $\alpha=0.9$ with various N.}
    \begin{tabular}[c]{ccc}
    \toprule
   {N}& {TG} &{$\alpha=0.9$} \\
    \midrule
    80 &0.0365&0.0383\\
    160&0.0364&0.0378\\
    320&0.0352&0.0373\\
    \bottomrule
    \end{tabular}
    \label{tab1}
\end{table}

Deconvolution problem is only a simple linear problem, in practice, the inverse problem should be more complex and ill-posed. Thus, we will use the HFTG prior to deal with more difficult problems in the next example.
%
%
\subsubsection{A Inverse source identification problem}\label{sec4.2.2}
In this example, we consider the source identification problem. Given the following initial-boundary value problem for the homogeneous heat equation
\begin{equation}\label{problem2}
\begin{split}
\begin{cases}
&\frac{\partial u(x,t)}{\partial t}=\Delta u(x,t)+f(x), \ (x,t)\in\Omega\times(0,T],\\
&u(x,t)=0, \ (x,t)\in\partial\Omega\times(0,T],\\
&u(x,0)=\varphi(x), \ x\in\Omega,\\
\end{cases}
\end{split}
\end{equation}
The corresponding inverse problem is to determine the heat source $f$ from the final temperature measurement $u(x,T)|_{x\in\Omega}$ with $\Omega=[1,3]$ and $T=1$. The initial temperature is given by
$$u(x,0)=\sin(\pi x)=\varphi(x), \ \ x\in[1,3],$$
and the heat source defined by
$$ f(x)=\left\{
\begin{array}{ll}
5,  &{1.15\leqslant x\leqslant1.35;}\\
5[\sin (6\pi x+\frac{\pi}{2})+1],  &{1.5\leqslant x \leqslant 2.5;}\\
5,  &{2.65\leqslant x\leqslant2.85;}\\
0,  &{otherwise.}
\end{array} \right. $$

We first solve the direct problem through the finite difference method (FDM) in \cite{L.Yan2010}. In keeping with the discrete schemes in section \ref{sec3.3}, assume $x=\exp(s)$ with $x\in[a,b]$, then $s\in [\ln(a),\ln(b)]$. The problem \eqref{problem2} can rewrite as
\begin{equation}\label{rewrite1}
\begin{split}
\begin{cases}
&\frac{\partial u(\exp(s),t)}{\partial t}=\exp(-2s)\left[\Delta u(\exp(s),t)-\frac{\partial u(\exp(s),t)}{\partial s}\right]\\
&\quad \quad\quad\quad\quad\quad\quad
  +f(\exp(s)), \  (\exp(s),t)\in\Omega\times(0,T],\\
&u(\exp(s),0)=\varphi(\exp(s)), \ \exp(s)\in\Omega,\\
&u(\exp(s),t)=0, \ (\exp(s),t)\in\partial\Omega\times(0,T],\\
\end{cases}
\end{split}
\end{equation}
Applying the same ideas discretize the problem \eqref{rewrite1} on a uniform grid in $s$ with the Crank-Nicolson method and the notations in \cite{L.Yan2010},
\begin{align*}
\frac{u(\exp(s),t+\Delta t)-u(\exp(s),t)}{\Delta t}=&\exp(-2s)\left[\theta_0 (\Delta u|_{t+\Delta t})+(1-\theta_0)(\Delta u|_t)\right]\\
&-\exp(-2s)\left[\theta_1 (D u|_{t+\Delta t})+(1-\theta_1)(D u|_t))\right]+f,
\end{align*}
where $0\leqslant\theta_0, \theta_1\leqslant1$, $\Delta t$ is the equally stepsize of time.
$\Delta u$ and $Du$ are both discretized by second order cental difference scheme.

Then the inverse problem of \eqref{rewrite1} has been reduced to solving the following linear system
\begin{equation}\label{probleminverse}
Af=b.
\end{equation}
The observed data $y$ are subject to noise, thus we have
$$y=Af+\eta,$$
where $\eta$ is the Gaussian observed noise with $\eta\thicksim \mathcal{N}(0, 0.001^2)$.

\begin{figure}[ht]
\centering
\subfigure[]{
\label{fig2:subfig1}
\includegraphics[width=7.5cm]{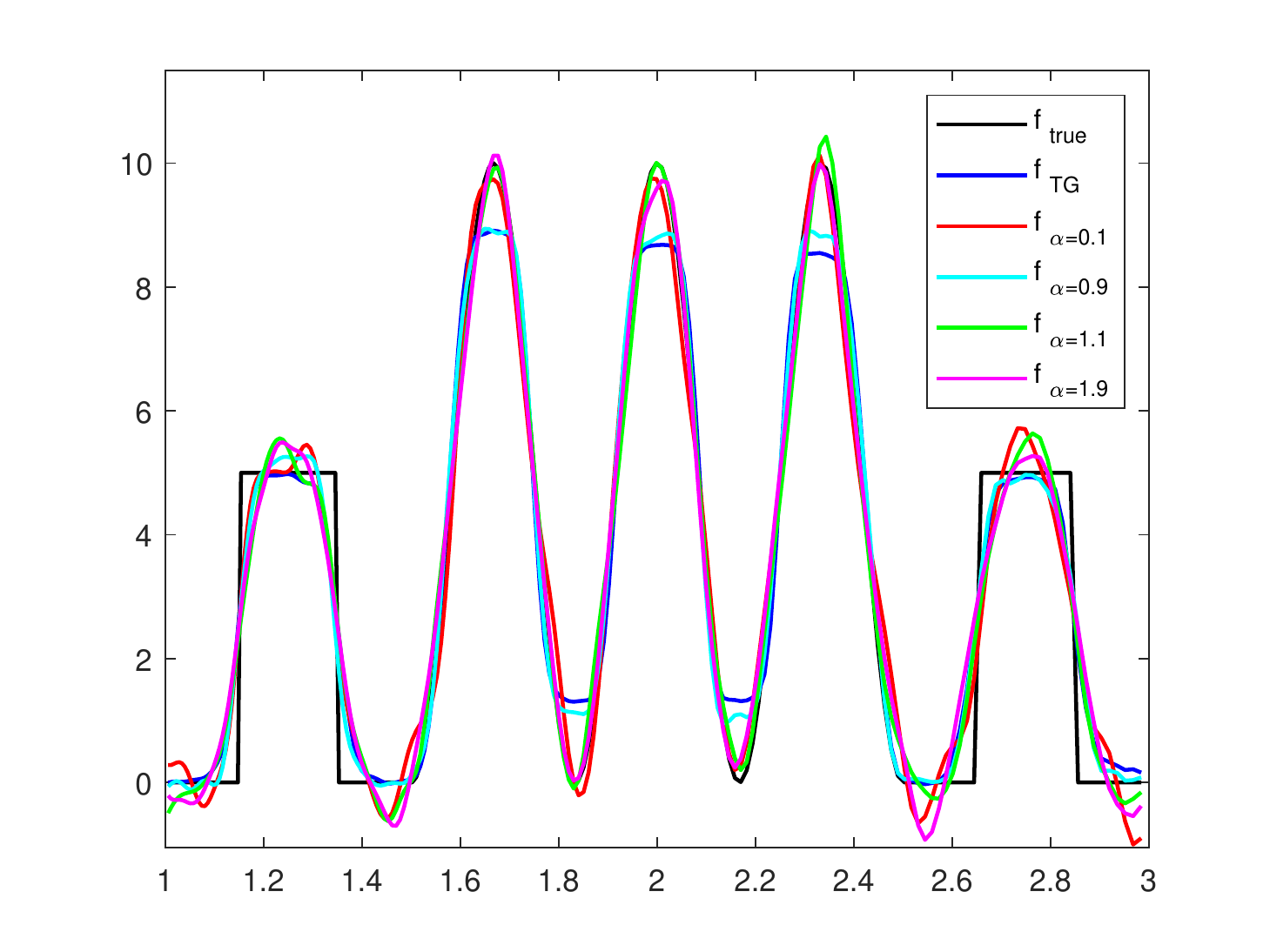}
}
\quad
\subfigure[]{
\label{fig2:subfig2}
\includegraphics[width=7.5cm]{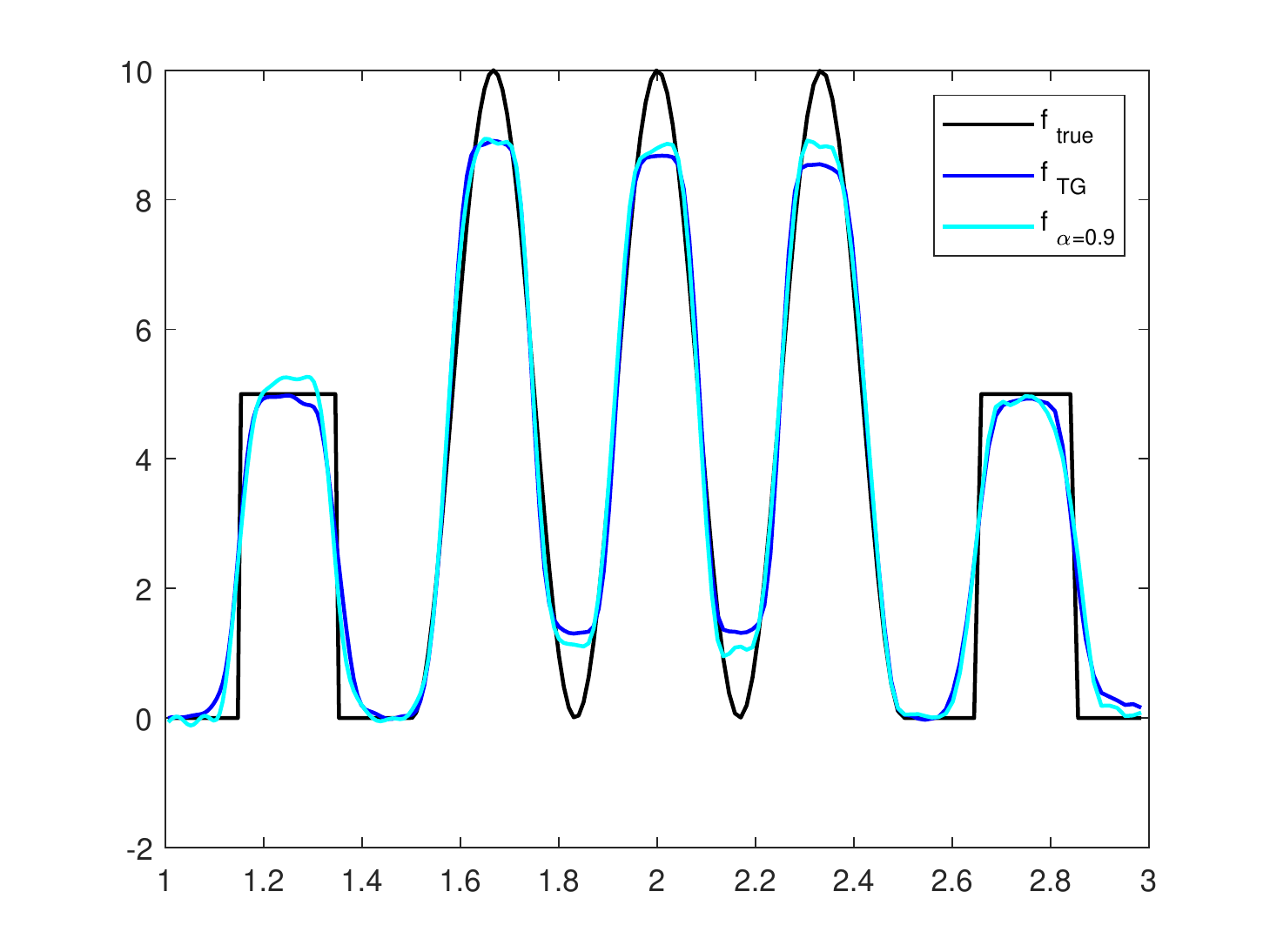}
}
\caption{(a): The true solution and inversion solution: The legend $f_{TG}$ represents TG prior inversion results, legend $f_{true}$ represents the true solution, and the others represents HFTG prior inversion results with different fractional oreder $\alpha$. (b): HFTG prior with $\alpha=0.9$ compared with TG prior.}
\label{figure2}
\end{figure}

We take $\theta=\theta_{0}=\theta_1=\frac{1}{2}$ and the number of uniform grids discrete in space and time is $ M = 200 $ and $ N = 120 $, respectively. For the HFTG prior, the discretization of the Hadamard fractional derivative is the same as equation \eqref{priordiscret1} and \eqref{priordiscret2} with $0<\alpha<1$ and $1<\alpha<2$. For the reference Gaussian prior measure, the covariance is same as equation \eqref{cov} with $\gamma=0.5$ and $d=0.03$. In order to ensure the reliability of the inference, we draw $10^6$ samples from the posterior measure and $2.5\times 10^5$ samples are used in the burn-in period with $\beta=0.02$ for pCN algorithm.

The numerical results for reconstructing the heat source $f$ with different fractional order $\alpha$ calculated by the HFTG prior and TG prior are shown in Figure \ref{figure2}. For the HFTG prior, when $\alpha=0.1,\ 0.9,\ 1.1$, and $1.9$, the corresponding regularization parameters are $\lambda=0.001,\ 0.06,\ 0.008$, and $0.0002$ respectively, and for the TG prior, we choose $\lambda=0.08$. The parameters in the Figure \ref{fig2:subfig2} are the same as in the Figure \ref{fig2:subfig1} with $\alpha=0.9$ and $f_{TG}$.

From the Figure \ref{figure2}, we can see that the advantages of different priors are more clear. For the TG prior, the results suffers from the staircase artifact in smooth due to the fact that the TV is local operator, but well approximate the flat. Nevertheless, the reconstruction with HFTG priors can overcome the weakness of TG prior because of that the HFTV is a non-local operator, but have blurry effect on the edges since it is less sensitive to edge than TV. For $\alpha = 0.9$, the result of HFTG prior is basic consistent with that of TG prior and the others are smoother than TG prior, which also similar to deconvolution problem \ref{sec4.2.1}.


This example is a linear problem but with a complex reconstruction truth, so that the inverse results are slightly worse than the first example. Also, we can see that the results of all methods for different priors are agree better with the true solution, which shows that all the priors of Bayesian inference methods are behaved well. The next we will consider a nonlinear inverse problem, which should be more ill-posed, to further appraise the behaviour of the HFTG prior.

%
%
\subsubsection{The parameter identify by interior measurement problem}\label{sec4.2.3}
In this example, we consider the nonlinear problem of identifying the parameter $q$ in the Dirichlet boundary value problem as following
\begin{equation}\label{problem3}
\left\{
\begin{array}{ll}
-\Delta u+qu=f,  &{in\ \Omega,}\\
u=0,  &{on\ \partial \Omega.}\\
\end{array} \right.
\end{equation}
The backward problem: Giving the source $f$, recover the coefficient $q$ from the measurements of the interior Neumann value $g=\frac{\partial u}{\partial n} |_{\Omega \backslash \partial \Omega}$. Similar to \cite{Gu2021}, we can define a nonlinear forward operator $G$ with $G(q)=g$. In this numerical example, we take $\Omega=[1,3]$, and the source $f$ is given by
$$f(x)=q(x)(x-1)(x-3)-2,$$
and the true solution $q$ of the inverse problem is a piecewise smooth function, defined as following
\begin{equation*}
q(x)=\left\{
\begin{array}{ll}
0.8,  &{1.3\leqslant x<1.6,}\\
1.4,  &{1.6\leqslant x<1.8,}\\
13(x-1.8)(x-2.2)+1.4,  &{1.8\leqslant x<2.2,}\\
1.4,  &{2.2\leqslant x<2.4,}\\
0.8,  &{2.4\leqslant x<2.7,}\\
0,  &{otherwise.}
\end{array} \right.
\end{equation*}
Applying the same way to section \ref{sec4.2.2}, and do the transformation $x=\ln(s)$ for problem \eqref{problem3} and then apply the equidistance discretization in variable $s$ with $N=200$ and the second order centered difference scheme to the equation after relevant transformation of problem \eqref{problem3}. Thus, the forward problem can be solved by the FDM and the observed data $y$ will be generated by the synthetic exact data $G(q)$ added the observed Gaussian noise $\eta$, i.e.,
$$y=G(q)+\eta.$$
\begin{figure}[ht]
\centering
\subfigure[]{
\label{fig3:subfig1}
\includegraphics[width=7.5cm]{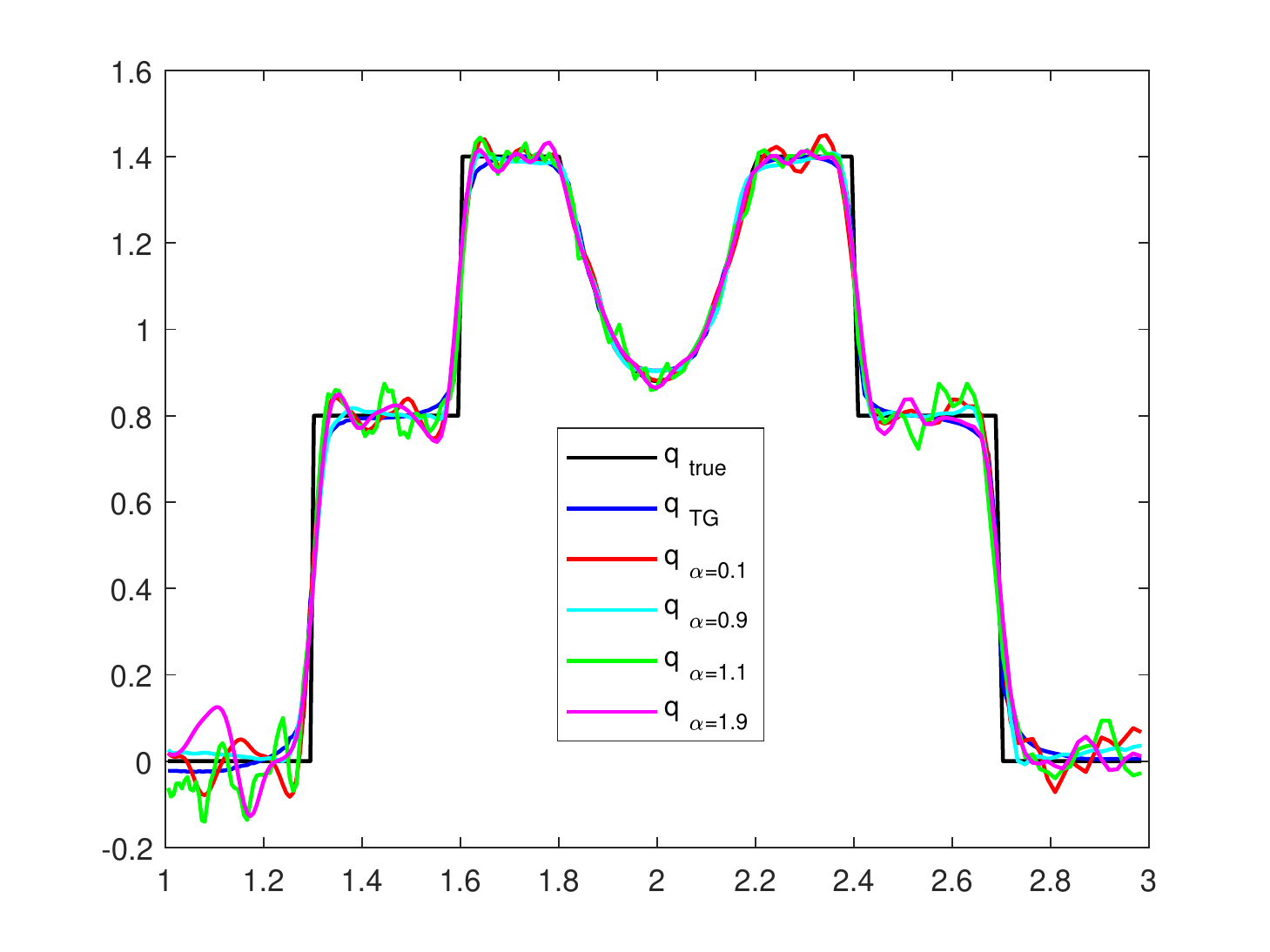}
}
\quad
\subfigure[]{
\label{fig3:subfig2}
\includegraphics[width=7.5cm]{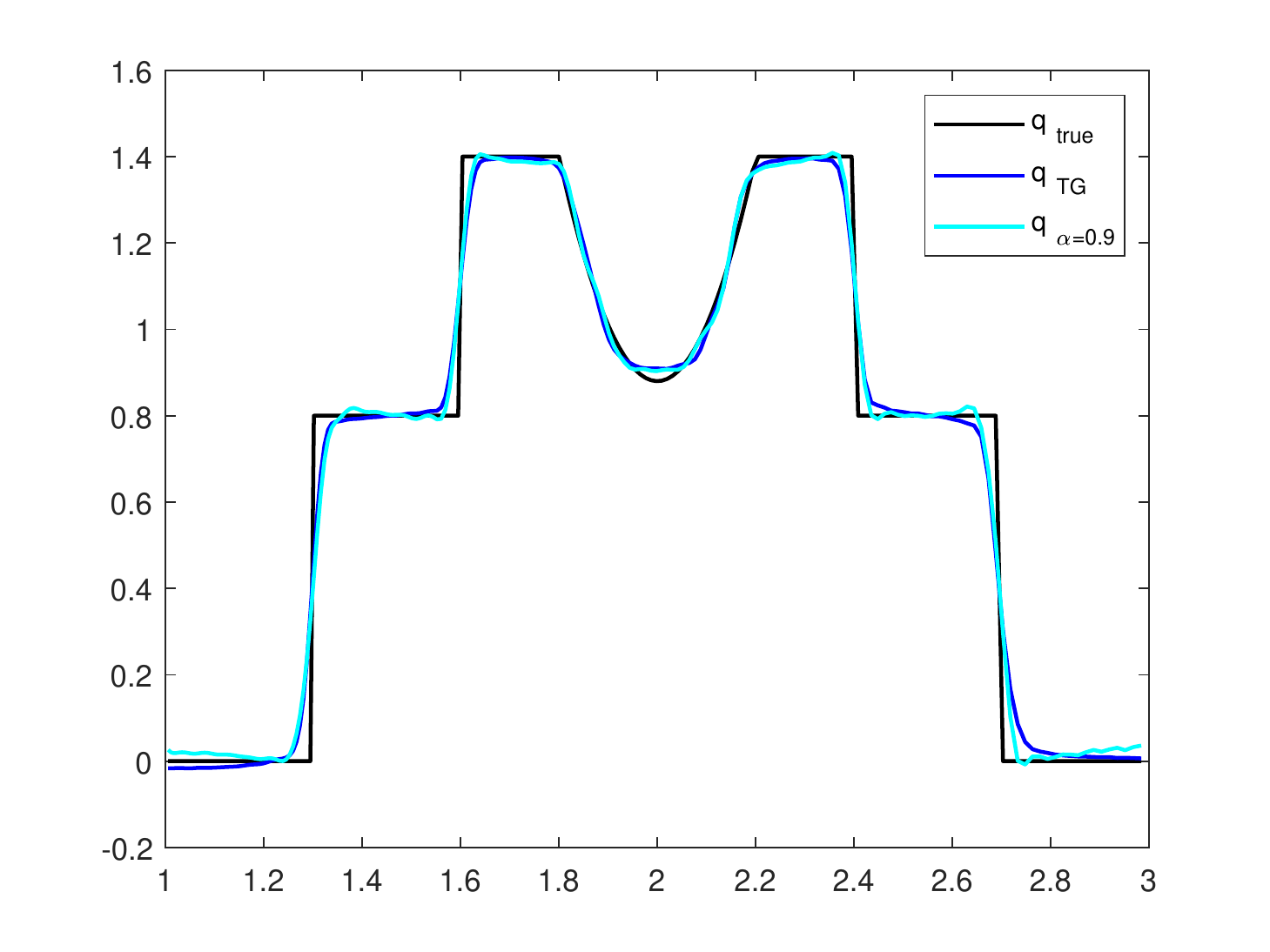}
}
\caption{(a): The true solution and inversion solution: The legend $f_{TG}$ represents TG prior inversion results, legend $f_{true}$ represents the true solution, and the others represents HFTG prior inversion results with different fractional oreder $\alpha$. (b): HFTG prior with $\alpha=0.9$ compared with TG prior.}
\label{figure3}
\end{figure}
For the HFTG prior, the discretization of the Hadamard fractional derivative is the same as equation \eqref{priordiscret1} and \eqref{priordiscret2} with $0<\alpha<1$ and $1<\alpha<2$. For the reference Gaussian prior measure, the covariance is same as equation \eqref{cov} with $\gamma=0.01$ and $d=0.03$.

In this numerical simulations, we take the noise as $\eta\thicksim \mathcal{N}(0,0.001^2)$. For the HFTG prior, the discretization of the Hadamard fractional derivative is the same as equation \eqref{priordiscret1} and \eqref{priordiscret2} with $0<\alpha<1$ and $1<\alpha<2$, and for the Gaussian measure, the covariance is again given by equation \eqref{cov} with $\gamma=1$ and $d=0.04$. We choose to draw $10^5$ samples from the posterior with pCN algorithm and set $\beta=0.01$ in Algorithm \ref{algorithm}.

The numerical results for reconstructing the coefficient $q$ with different fractional order $\alpha$ calculated by the HFTG prior and TG prior are shown in Figure \ref{figure3}. For the HFTG prior, when $\alpha=0.1,\ 0.9,\ 1.1$, and $1.9$, the corresponding regularization parameters are $\lambda=0.08,\ 2,\ 0.05$, and $0.001$ respectively, and for the TG prior, we choose $\lambda=2$. The parameters in the Figure \ref{fig3:subfig2} are the same as in the Figure \ref{fig3:subfig1} with $\alpha=0.9$ and $f_{TG}$.

The Figure \ref{figure3} shows that the HFTG prior is behaved well in smooth piece and a little worst in discontinuous pieces compared with TG prior. However, in Figure \ref{fig3:subfig2}, when $\alpha=0.9$, the results of HFTG prior and TG prior are roughly the same. This features are showing no difference with the previous instances. In addition, we can see that the results of all the different prior can approximate the true function, indicating the posterior distributions derived by all the prior are well behaved. Thus, this suggests that the HFTG prior is feasible and reasonable.

%
%
\section{Inclusion}\label{sec5}
Throughout this paper, we investigate an infinite-dimensional Bayesian inference method based on a Hadamard fractional total variation-Gaussian (HFTG) prior. We first specify the fractional Sobolev space $W^{\alpha}_2$ as the space of unknown functions $X$, and the Hadamard fractional total variation as the additional regularization term, which construct the HFTG prior. Then the well-posedness and finite-dimensional approximation of the posterior measure of the Bayesian inversion with this prior has been obtained under the particular assumption for the forward operator and the property of fractional total variation. Finally, we use the pCN algorithm to implement the sampling from the posterior distribution in the Bayesian inference with respect to the HFTG prior. The numerical results show that the HFTG prior is effective and behaved well on avoiding step effects and capturing the detailed texture of the image. We believe that the HFTG prior can be used to many other inverse problems, such as scattering inverse problem and so on, which is our research interest in future.

%
%
\section{Acknowledgements}\label{sec6}
The work described in this paper was supported by the NSF of China (11301168) and NSF of
Hunan (2020JJ4166).

\section{Appendix}

\textbf{Proof of Theorem \ref{approximationtheorem}:}
\begin{proof}
Set $X=W_2^{\alpha}(\Omega)$, for $\forall$ $r>0$, $\exists$ $K_1=K_1(r)>0$, $K_2=K_2(r)>0$ such that, for all $u\in X$ with $\Vert u\Vert_X<r$, and satisfies:
$\Phi(u)\leq K_1$, $R(u)\leq K_2$. Throughout the proof, the constant C changes from occurrence to occurrence.
Let $r=\|y\|_\Sigma$ and $K(r)=K_1(r)+K_2(r)$, then it is easy to see that the normalization constant $Z$ for $\mu^y$ satisfies,
\[
Z\geqslant\int_{\lbrace\Vert u\Vert_X<r\rbrace}\exp(-K(r))\mu_0(\mathrm{d}u)=\exp(-K(r))\mu_0\lbrace\Vert u\Vert_X<r\rbrace=C.
\]
Similarly, one can show the normalization constant for $\mu^y_{N_1,N_2}$ also satisfies $Z_{N_1,N_2}\geq C$.
Moreover, for $\forall\ \varepsilon\in(0,1)$,
\begin{align*}
\vert Z-Z_{N_1,N_2}\vert &\leqslant \int_X\vert\exp(-\Phi(u)-R(u))-\exp(-\Phi_{N_1}(u)-R_{N_2}(u))\vert\mathrm{d}\mu_0(u)\\
&\leqslant\int_{X\backslash X_\varepsilon}\mu_0(\mathrm{d}u)+\int_{X_\varepsilon}\vert\Phi(u)-\Phi_{N_1}(u)\vert\mu_0(\mathrm{d}u)+\int_{X_\varepsilon}\vert R(u)-R_{N_2}(u)\vert\mu_0(\mathrm{d}u)\\
&\leqslant\varepsilon+a_{N_1}(\varepsilon)+b_{N_2}(\varepsilon).
\end{align*}

By the definition of Hellinger distance, it finds
\begin{equation}
\begin{aligned}
2d_\mathrm{Hell}(\mu^y,\mu^y_{N_1,N_2})^2 &= \int_X \left(\sqrt{\frac{d\mu^y}{d\mu_0}}-\sqrt{\frac{d\mu^y_{N_1,N_2}}{d\mu_0}}\right)^2 \mu_0(du)
\\&=\int_X(Z^{-\frac12}\exp(-\frac12\Phi(u)-\frac12R(u))-Z_{N_1,N_2}^{-\frac12}\exp(-\frac12\Phi_{N_1}(u)-\frac12R_{N_2}(u)))^2\mu_0(\mathrm{d}u)
\\ &\leqslant I_1+I_2+I_3,
\end{aligned}
\end{equation}

where

\begin{equation*}
\begin{aligned}
&I_1=\int_{X\backslash X_\varepsilon}(\frac1{\sqrt{Z}}\exp(-\frac12\Phi(u)-\frac12R(u))-\frac1{\sqrt{Z_{N_1,N_2}}}\exp(-\frac12\Phi_{N_1}(u)-\frac12R_{N_2}(u)))^2\mu_0(\mathrm{d}u),\\
&I_2=\frac2Z\int_{X_\varepsilon}(\exp(-\frac12\Phi(u)-\frac12R(u))-\exp(-\frac12\Phi_{N_1}(u)-\frac12R_{N_2}(u)))^2\mu_0(\mathrm{d}u),\\
&I_3=2\vert Z^{-\frac12}-(Z_{N_1,N_2})^{-\frac12}\vert^2\int_{X_\varepsilon}\exp(-\Phi_{N_1}(u)-R_{N_2}(u))\mu_0(\mathrm{d}u).
\end{aligned}
\end{equation*}

Actually, we can show
\begin{equation*}
\begin{aligned}
&I_1\leqslant\int_{X\backslash X_\varepsilon}(2C^{-\frac12})^2\mu_0(\mathrm{d}u)\leq C\varepsilon,\\
&I_2\leqslant\frac2C\int_{X_\varepsilon}(a_{N_1}(\varepsilon)+b_{N_2}(\varepsilon))^2\mu_0(\mathrm{d}u)\leq C(a_{N_1}(\varepsilon)+b_{N_2}(\varepsilon))^2,\\
&I_3\leqslant C(Z^{-3}\wedge(Z_{N_1,N_2})^{-3})\vert Z-Z_{N_1,N_2}\vert^2\int_{X_\varepsilon}\mu_0(\mathrm{d}u)=C(\varepsilon+a_{N_1}(\varepsilon)+b_{N_2}(\varepsilon))^2.
\end{aligned}
\end{equation*}

It implies that
\[
2d^2_\mathrm{Hell}(\mu^y,\mu^y_{N_1,N_2})\leq C(\varepsilon+\varepsilon^2+(a_{N_1}(\varepsilon)+b_{N_2}(\varepsilon))^2+\varepsilon(a_{N_1}(\varepsilon)+b_{N_2}(\varepsilon))),
\]
here $C$ is a constant independent of $N_1,N_2$. Let $N_1$, $N_2\rightarrow+\infty$, we have that
\[
\lim_{N_1,N_2\to+\infty}2d_\mathrm{Hell}(\mu,\mu_{N_1,N_2})^2\leq C(\varepsilon+\varepsilon^2),
\]
for any $\varepsilon>0$. Hence,
\[
\lim_{N_1,N_2\to+\infty}d_\mathrm{Hell}(\mu^y,\mu^y_{N_1,N_2})=0.
\]
\end{proof}

\textbf{Proof of Corollary \ref{approximationcor}:}
\begin{proof}
 Set $X=W^{\alpha,\psi}_{2}(\Omega)$ and
\[
a_N=\mathbb{E}\| u-u_N\|_X^2= \sum_{k=N+1}^\infty \mathbb{E}|\langle u,e_k\rangle|^2.
\]
Notice $\mathcal{C}_0$ is in the trace class, then $a_N\to 0$ as $N\to\infty.$
By using Markov's inequality, it finds, for $\forall\ \varepsilon>0$ and $\forall\ N\in \mathbb{N}$,
\begin{equation}
\mu_0(\lbrace\Vert u-u_N\Vert_X>\sqrt{\frac{2a_N}{\epsilon}}\rbrace)\leq \frac12\varepsilon. \label{e:eq1}
\end{equation}

For above $\varepsilon$, $\exists\ r_\varepsilon$ such that $\mu_0(\{u\in X\, |\, \Vert u\Vert_X>r_\varepsilon\})<\frac12\varepsilon.$
Clearly, for $\forall\ N\in\mathbb{N}$,
\[
\mu_0(\lbrace u\in X\, |\, \Vert u\Vert_X\leq r_\varepsilon,\,\Vert u-u_N\Vert_X\leq\sqrt{\frac{2a_N}{\varepsilon}}\rbrace)\geq 1-\varepsilon.
\]
Setting $\widetilde{X}=\lbrace u\in X\, |\, \Vert u\Vert_X\leq r_\varepsilon,\,\Vert u-u_N\Vert_X\leq\sqrt{\frac{2a_N}{\varepsilon}}\rbrace.$

Since $G$ satisfies the Assumptions \ref{forward operator assume}, the $\Phi$ satisfies Assumptions 2.6 in \cite{Stuart2010} and $R$ defines by Equation \eqref{R definition}, it follows that there are constants $L^\Phi_\varepsilon, L^R_\varepsilon>0$, such that for $\forall\ u\in \widetilde{X}$,
\begin{gather*}
\vert\Phi(u)-\Phi(u_N)\vert\leq L^\Phi_\varepsilon\Vert u-u_N\Vert_X\leq L^\Phi_\varepsilon\sqrt{\frac{2a_N}{\varepsilon}},\\
\vert R(u)-R(u_N)\vert\leq L^R_\varepsilon\Vert u-u_N\Vert_X\leq L^R_\varepsilon\sqrt{\frac{2a_N}{\varepsilon}}.
\end{gather*}
it is easy to see $L^\Phi_\varepsilon\sqrt{\frac{2a_N}{\varepsilon}}, L^R_\varepsilon\sqrt{\frac{2a_N}{\varepsilon}}\to0$ as $N\to\infty$.
As,
 \[\widetilde{X}\subset X_\varepsilon=\{u\in X\, |\,\vert\Phi(u)-\Phi(u_N)\vert\leq L^\Phi_\varepsilon\sqrt{\frac{2a_N}{\varepsilon}}, \,
\vert R(u)-R(u_N)\vert\leq L^R_\varepsilon\sqrt{\frac{2a_N}{\varepsilon}}\},\]
it deduces $\mu_0(X_\varepsilon)\geq 1-\varepsilon$,
and by using Theorem \ref{approximationtheorem},
\[
d_\mathrm{Hell}(\mu^y,\mu^y_N)\to0,~~\mathrm{as}~~N\to\infty.
\]
\end{proof}

\end{document}